\documentclass[12pt,english,reqno]{amsart}
\usepackage{amscd,amssymb,amsmath,amstext,amsthm,amsfonts,bbold}
\usepackage[dvips]{graphicx}
\usepackage{mathrsfs}
\usepackage[normalem]{ulem}
\usepackage{enumerate}
\usepackage{float}
\usepackage{hyperref}
\usepackage[dvipsnames]{xcolor}
\usepackage[ansinew]{inputenc}

\usepackage{a4wide}
\newtheorem{maintheorem}{Theorem}

\newtheorem{maincorollary}[maintheorem]{Corollary}

\newtheorem{T}{Theorem}[section]

\newtheorem{Corollary}[T]{Corollary}

\newtheorem{Lemma}[T]{Lemma}

\newtheorem{Remark}[T]{Remark}

\newtheorem{Definition}[T]{Definition}

\newtheorem{Claim}{Claim}

\def \RR {{\mathbb R}}

\def \NN {{\mathbb N}}

\def \XX {{\mathbb X}}

\def \cl {\mathcal{L}}

\def \cm {\mathcal{M}}

\def \cd {\mathcal{D}}

\def \cp {\mathcal{P}}

\def \cu {\mathcal{U}}
\def \co {\mathcal{O}}
\def \cn {\mathcal{N}}

\def \cR {\mathscr{R}}

\newcommand{\leb}{\operatorname{Leb}}
\newcommand{\dist}{\operatorname{dist}}

\newcommand{\supp}{\operatorname{supp}}
\newcommand{\per}{\operatorname{Per}}

\newcommand{\diameter}{\operatorname{diameter}}

\begin{document}

\thanks{Work carried out at Federal University of
Bahia (UFBA) and at State University of Feira de Santana (UEFS). Partially supported by CNPq-Brazil (PQ 313272/2020-4)}

\author{F Pedreira}
\address{DEXA, UEFS, Campus Universitário, 44100-000 Feira de Santana Bahia, Brazil}
\email{fopedreira@uefs.br}

\author{V Pinheiro}
\address{Instituto de Matematica - UFBA, Av. Ademar de Barros, s/n,
40170-110 Salvador Bahia, Brazil}
\email{viltonj@ufba.br}




\date{\today}

\title{Super-expanding measures}

\maketitle

\begin{abstract}
We study the one-dimensional expanding Lorenz maps and show the existence of dense subset $\mathcal{D}$ of Lorens maps such that each $f\in\mathcal{D}$ has an uncountable set of ergodic invariant probabilities with infinite Lyapunov exponent and positive entropy.
Such measures may appear when the singularity has fast recurrence to itself.
Conversely, if the singularity has slow recurrence to itself then the Lorenz map has an upper bound to the Lyapunov exponent of all invariant measures.
\end{abstract}

\tableofcontents
\section{Introduction}\label{Introduction}

One of the most impactful works in the area of Dynamic Systems was with the studies of mathematician and meteorologist Edward Lorenz, published in the Journal of Atmospheric Sciences \cite{L} in 1963. Motivated by an attempt to understand the fundamentals of weather forecasting, he obtained a model for the convection of thermal fluids, given by the system of differential equations
\vspace{-.35cm}
\begin{eqnarray}\label{sistequationsLorenz}
 \dot{x}& = & -\sigma x + \sigma y \nonumber   \\
 \dot{y}& = & \rho x - y -xz    \\
 \dot{z}& = & -\beta z + xy  \nonumber 
\end{eqnarray}
for parameters $\sigma=10$, $\rho =28$ and $\beta = 8/3$. 

The behavior observed by him in system \eqref{sistequationsLorenz}, originated what is now known as a strange attractor, and led many researchers to suggest a geometric model for the Lorenz attractor.
Among them, Afraimovich, Bykov and Shil'nikov \cite{ABS} in 1977, Guckenheimer and Williams \cite{GW,Wi} in 1979, presented a construction of the model, dynamically similar to that of Lorenz, in a  linearized neighborhood, whose origin is a singularity with eigenvalues $\lambda_2 < \lambda_3 < 0 < \lambda_1$ and with expanding condition $\lambda_1 + \lambda_3 >0$ (see Figure~\ref{LorenzFluxo}). 

\begin{figure}
  \begin{center}\includegraphics[scale=.2]{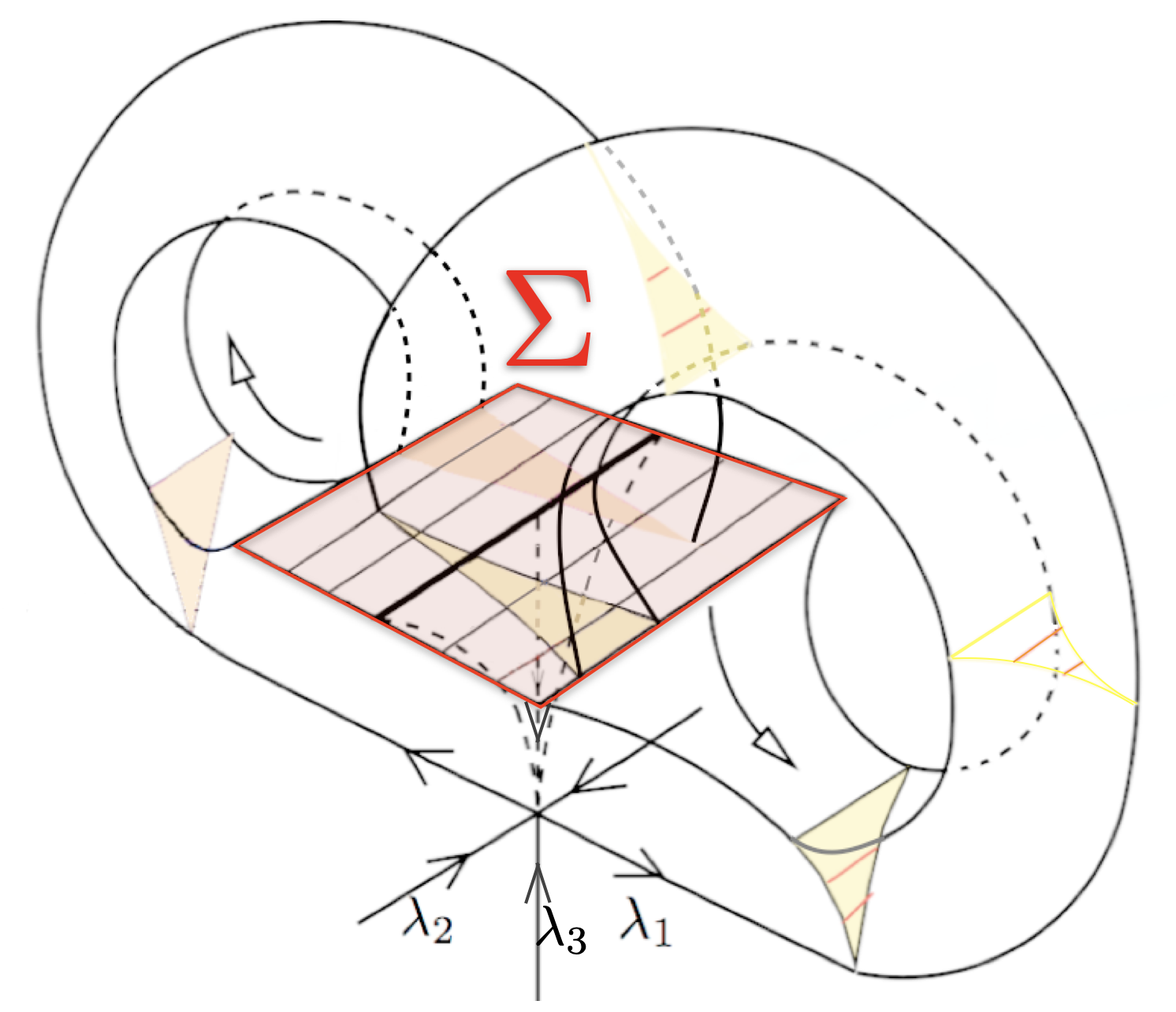}\hspace{0.8cm}\includegraphics[scale=.3]{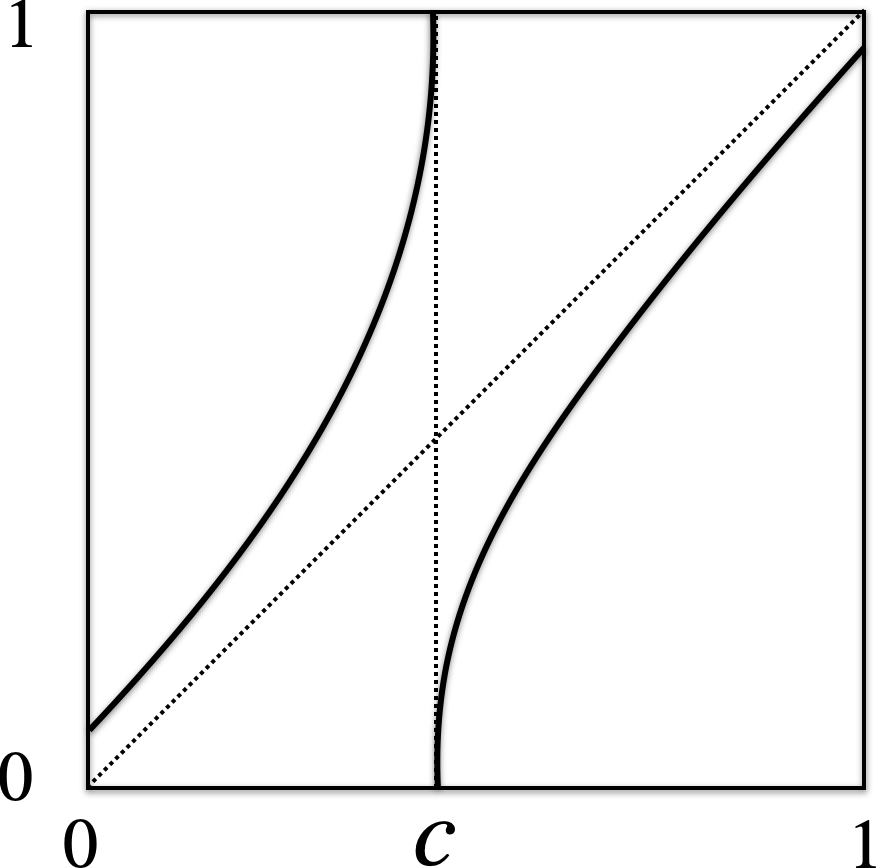}\hspace{1cm}
  \caption{\small On the left side one can see the geometric model of the Lorenz Flow. On the right, we have an expanding Lorenz map with a singularity $c$.}\label{LorenzFluxo}
  \end{center}
\end{figure}

The dynamical behavior is then analyzed by taking a Poincaré first return map $P$, to a cross-section to the flow.
The map $P$ induces a interval map $f$  with a single point of discontinuity,  which is a singularity (see Figure~\ref{LorenzFluxo}).
This type of map is called an {\bf\em expanding Lorenz map}. 
Due to this reduction to an interval dynamics, the properties of the  Geometric Lorenz Attractors have been extensively studied (see for instance in \cite{AP, GH, Sp}).
Furthermore, Tucker proved in \cite{T} that the original Lorenz Attractor is in fact a Geometric Lorenz Attractor.

Although in many aspects the expanding Lorenz maps are well understood, the presence of a singularity add, as shown here, one unexpected feature: the existence of non-periodic super-expanding measures.

When $f$ is an interval map, the {\bf\em Lyapunov exponent}  of a point $x$ is defined as $$\lambda_f(x):=\lim_{n\to+\infty}\frac{1}{n}\log|(f^n)'(x)|=\lim_{n\to+\infty}\frac{1}{n}\sum_{j=0}^{n-1}\log|f'\circ f^j(x)|,$$ whenever this limit exists, otherwise one can consider the upper Lyapunov exponent $\limsup_{n\to+\infty}\frac{1}{n}\sum_{j=0}^{n-1}\log|f'\circ f^j(x)|$.
By Birkhoff, if $\mu$ is an ergodic invariant probability, $\lambda_f(x)$ $=$ $\int_{x\in[0,1]}\log|f'(x)|d\mu$
for $\mu$ almost every point $x$.
Thus, the {\bf\em Lyapunov exponent} of $\mu$ is  $\lambda_f(\mu)=\int\log|f'|d\mu$.
A {\bf\em super-expanding measure} is an ergodic invariant probability  having $+\infty$ as its Lyapunov exponent.

According to Przytycki (Theorem~B in \cite{Pr}), if $f$ is an interval maps and $\mu$ is an ergodic invariant probability having $-\infty$ as its Lyapunov exponent then $\mu$ must be supported on a periodic critical orbit.
In particular, the metric entropy $h_{\mu}(f)$ is zero.
For an expanding Lorenz map $f$ with a singular point $c$ (as in Figure~\ref{LorenzFluxo}), we have that $\lambda_f(\mu)>0$ for every ergodic invariant probability $\mu$.
Furthermore, if the singularity of the Lorenz map is periodic, i.e.,  $f^{\ell}(c_{\pm})=c$ for some $\ell\ge1$ (\footnote{ Given a point $p$ in the interval and $n\ge1$, define  $f'(p_-)=\lim_{x\uparrow p}f'(x)$, $f'(p_+)=\lim_{x\downarrow p}f'(x)$, $f^{n}(p_{-})=\lim_{x\uparrow p}f^{n}(x)$ and $f^{n}(p_{+})=\lim_{x\downarrow p}f^{n}(x)$.}), then $\lambda_f(\mu)=+\infty$ for $\mu=\frac{1}{\ell}\sum_{j=0}^{\ell-1}\delta_{f^j(c_{\pm})}$.
Of course, in this case, $\mu$ has a finite support and zero entropy $($\footnote{ One can extend $f$ to $c$ so that $\mu$ becomes an ergodic invariant probability. For that, set $f(c):=f(c_-)$ if $f^{\ell}(c_-)=c$, otherwise set  $f(c):=f(c_+)$.}$)$.
However, there is no Przytycki-like result  having $+\infty$ as the Lyapunov exponent and so, the existence of ergodic probabilities with infinite Lyapunov exponent and positive entropy have been conjectured for many years.

In \cite{Do}, Dobbs gives two types of maps $g_0,g_1:[0,1]\to[0,1]$ presenting super-expanding measures.
In both cases the maps are conjugated the full tent map $T(x)=1-2|x-1/2|$, that is, $g_j=h_j^{-1}\circ T\circ h_j$ for some homeomorphism $h_j:[0,1]\to[0,1]$.
In the first case, $h_0\in C([0,1])\cap C^{\infty}((0,1))$ and $g_0$ has two singularities $0$ and $1$, $g_0'(0)=+\infty$, $g_0'(1)=-\infty$, and a criticality $g_0'(1/2)=0$ (see Figure~\ref{Dobbs.png}).
In the second case, $g_1$ has three singularities as illustrated on the right side of Figure~\ref{Dobbs.png}.
Roughly speaking, choosing a suitable $h_j$ as well as the order $\alpha$ of the singularity $0$ ($g_j(x)\approx x^{\alpha}$ for $x\approx0$), 
Dobbs has able to show that $\mu_{g_j}:=\leb\circ\,h_j$ is a super-expanding measure for $g_j$  with full topological entropy $h_{\mu_{g_j}}(g_j)=\log2=h_{top}(g_j)$.
In the first case, $\log|g_0'|$ was not integrable with respect to $\mu_{g_0}$ and 
\begin{equation}\label{EquationDoOlVi}
  -\infty=\liminf_{n\to+\infty}\frac{1}{n}\log|(g_0^n)'(x)|<\limsup_{n\to+\infty}\frac{1}{n}\log|(g_0^n)'(x)|=+\infty
\end{equation}
 $\mu_{g_0}$ almost everywhere.
In the second case, $\liminf_{n}\frac{1}{n}\log|(g_1^n)'(x)|=+\infty$ for $\mu_{g_1}$ almost every $x\in[0,1]$. 
\begin{figure}
  \begin{center}\includegraphics[scale=.38]{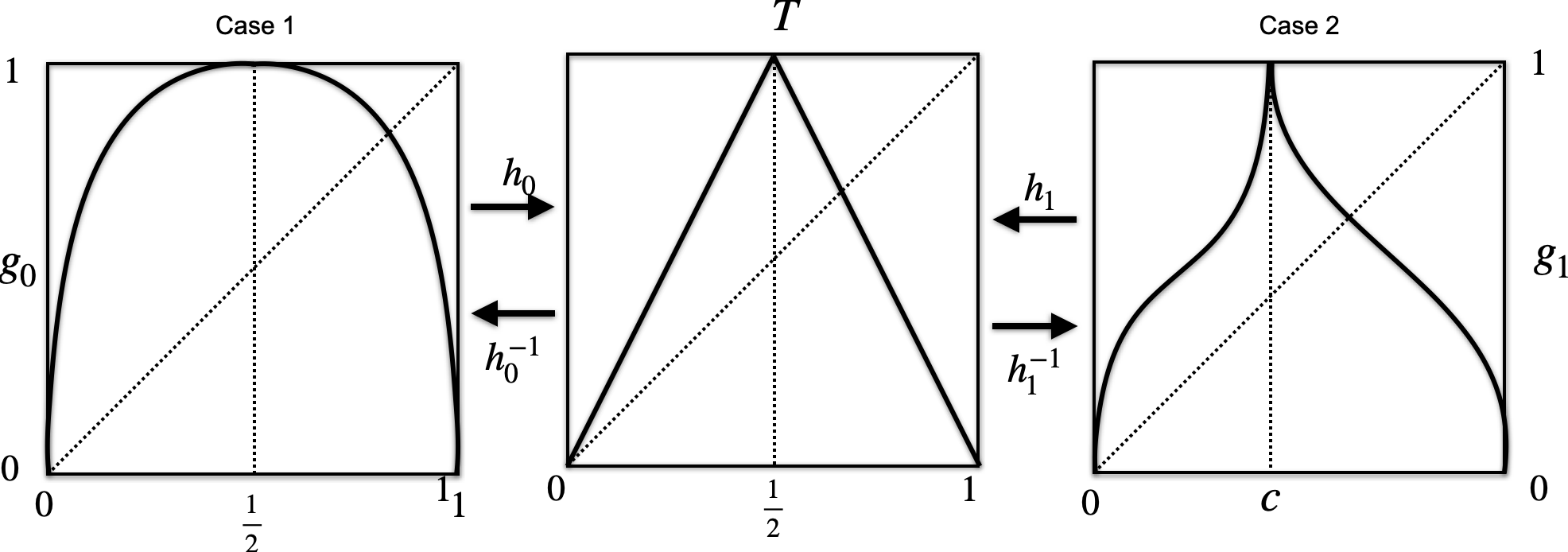}
  \caption{\small On the left side one can see a map $g_0$ with two singularities $0$ and $1$, and a critical point $c=1/2$, this map is conjugated by $h_0$ with the completed tent map $T$ on center of the picture.
  On the right side, the map $g_1$ is conjugated by $h_1$ to $T$ and it has three singularities: $0$, $c$ and $1$.
 According to \cite{Do}, $g_0$ satisfies  equation \eqref{EquationDoOlVi} $\mu_{g_{_0}}\hspace{-0.1cm}$-almost everywhere, where $\mu_{g_{_0}}\hspace{-0.2cm}=\hspace{-0.1cm}\leb\circ\,h_0$, and 
 $\lim_{n}\frac{1}{n}\log|(g_1^n)'(x)|=+\infty$ for $\mu_{g_1}\hspace{-0.15cm}=\leb\circ\,h_1$ almost every $x$.}\label{Dobbs.png}
  \end{center}
\end{figure}
Recently, Olivares-Vinales \cite{OV} consider a  unimodal map $g_0$ conjugated with a tent map with a Fibonacci recurrence of the turning point.
In \cite{OV}, $g_0$ has three non regular points, including a singularity (see Figure~\ref{Olivares-Vinales.png}). 
\begin{figure}
  \begin{center}\includegraphics[scale=.28]{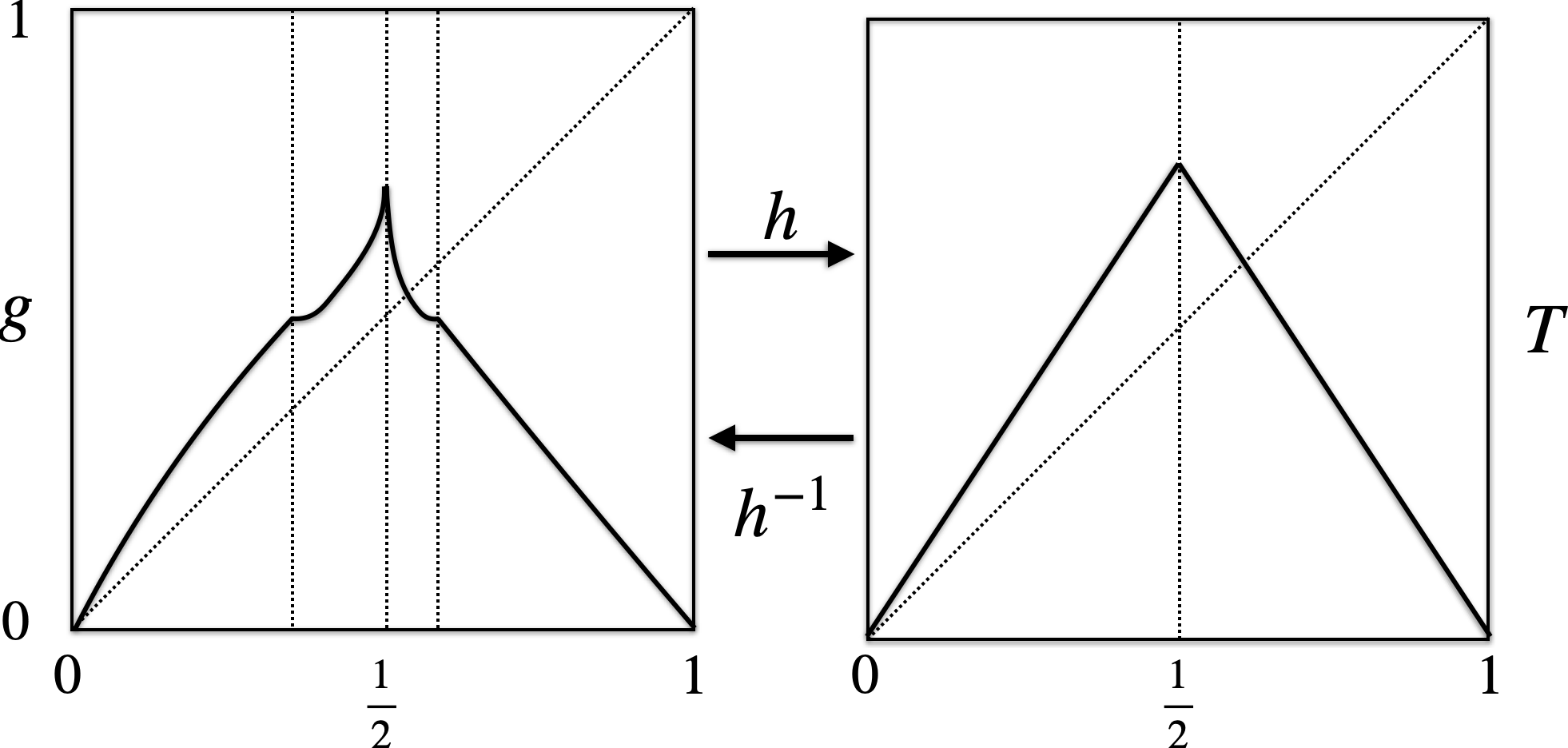}
  \caption{\small On the left side one can see a map $g_0$ with three non-regular points $p$, $q$  and the a singularity $c=1/2$.  The map $g_0$ is conjugated by $h$ with a ``Fibonacci'' tent map $T$, i.e., the turning point $1/2$ has a Fibonacci recurrence to itself. The tent map $T$ is represented on the right side of the figure. According to \cite{OV}, if $\mu_{_T}$ is the unique $T$ invariant probability supported on the minimal set $\overline{\co_T^+(1/2)}$ then $\mu_{g_{_0}}\hspace{-0.15cm}=\mu_{_T}\circ\, h$ is a super-expanding measure, satisfies \eqref{EquationDoOlVi}, with zero entropy.} \label{Olivares-Vinales.png}
  \end{center}
\end{figure}
With that, Olivares-Vinales shows that $\mu_{g_0}:=\mu_{_T}\circ\,h$ is an example of a non-periodic ergodic invariant probability with zero entropy and such that \eqref{EquationDoOlVi} holds for $\mu_{g_0}$-almost every point $x$,
where $h$ is the conjugation of  $g_0$ with the ``Fibonacci tent'' $T$ and $\mu_{_T}$ is unique $T$-invariant probability of the minimal set $\overline{\co_T^+(1/2)}$.

In the papers above, the authors used conjugations to produce, in each example, a super-expanding measure. 
In the present paper, for Lorenz maps, we show the existence of dense set $\mathcal{D}$ of Lorenz maps having uncountable many ergodic invariant probabilities with infinite Lyapunov exponent (super-expanding measures) and positive entropy.
Moreover, $\cd$ contains all Lorenz maps with a periodic singularity (see Section~\ref{SectionSMR}) and, for all maps in $\mathcal{D}$, the supreme of the entropy of the super-expanding measures realizes the topological entropy of the system. 
 
We observe that singularities appear naturally in many dynamical systems, for instance: non convex billiards.
Hence, this work may indicate that such an abundance of super-expanding measures with positive entropy can occur in other important dynamical systems.

\section{Statement of mains results}\label{SectionSMR}

As can be seen in Figure~\ref{doismapas}, one can extend the Lorenz map in the right side of the Figure~\ref{LorenzFluxo} to an interval map with two fixed points.
Because of that, we define an expanding Lorens map as follows.   
A $C^{1+}$ interval map $f:[0,1]\setminus\{c\}  \to [0,1]$, $0< c < 1$, is a called an {\bf\em expanding Lorenz map} if $f(0)=0, f(1)=1, f'(x)\ge \lambda > 1, \forall x \in [0,1]\setminus  \{c\}$ and $ f'(c_{\pm})=+\infty$.
The point $c$ in the definition above is called the {\bf\em singular point} or the {\bf\em singularity} of $f$ and its {\bf\em singular values} are $f(c_-)$ and $f(c_+)$.
An expanding Lorenz map $f$, with singularity $c$, is called {\bf\em non-flat} if there exist constants $\alpha, \beta \in (0,1), d_{0},d_{1} \in [0,1] $ and $C^{1+}$ diffeomorphisms $\phi_{0,f}:[0,c]\rightarrow [0,d_{0}^{1/\alpha}]$ and $\phi_{1,f}:[c,1]\rightarrow [0,d_{1}^{1/\beta}]$ such that, $\phi_{j,f}(c)=0$, $\phi_{0,f}(0)=(d_0)^{1/\alpha}$, $\phi_{1,f}(1)=(d_1)^{1/\beta}$ and 
\begin{equation}\label{Equationjytrdcvbjk}
   f(x) = \left \{ \begin{matrix} d_{0} - (\phi_{0,f}(x))^{\alpha} &
      \mbox{if }x <c,
      \\1- d_{1} + (\phi_{1,f}(x))^{\beta}  & \mbox{if }x > c. \end{matrix} \right.
\end{equation}
where  $f(c_{-})=d_{0}$ and  $f(c_{+})=1- d_{1}$.

A set $\Lambda\subset [0,1]$ satisfies the slow recurrence condition to $c$  if for each $\epsilon >0$ there is a $\delta >0$ such that 
\begin{equation}\label{slowrecurrencLambda}
\limsup_{n\to\infty}\dfrac{1}{n}\sum_{j=0}^{n-1}-\log\dist_{\delta}(f^{j}(x),c)\le \epsilon
\end{equation}
for every $x \in \Lambda $, where $\dist_{\delta}(x,c)$  denotes the $\delta$-truncated distance from $x$ to $c$ defined as $\dist_{\delta}(x,c)=\dist(x, c)  \hbox{ if }   \dist(x,c)\le\delta$, and $\dist_{\delta}(x,c)=1$ otherwise.
An ergodic $f$-invariant probability $\mu$ satisfies \textit{slow recurrent condition} if there is a set $\Lambda$ satisfying \eqref{slowrecurrencLambda} such that $\mu(\Lambda)=1$.

\begin{figure}
  \begin{center}\includegraphics[scale=.4]{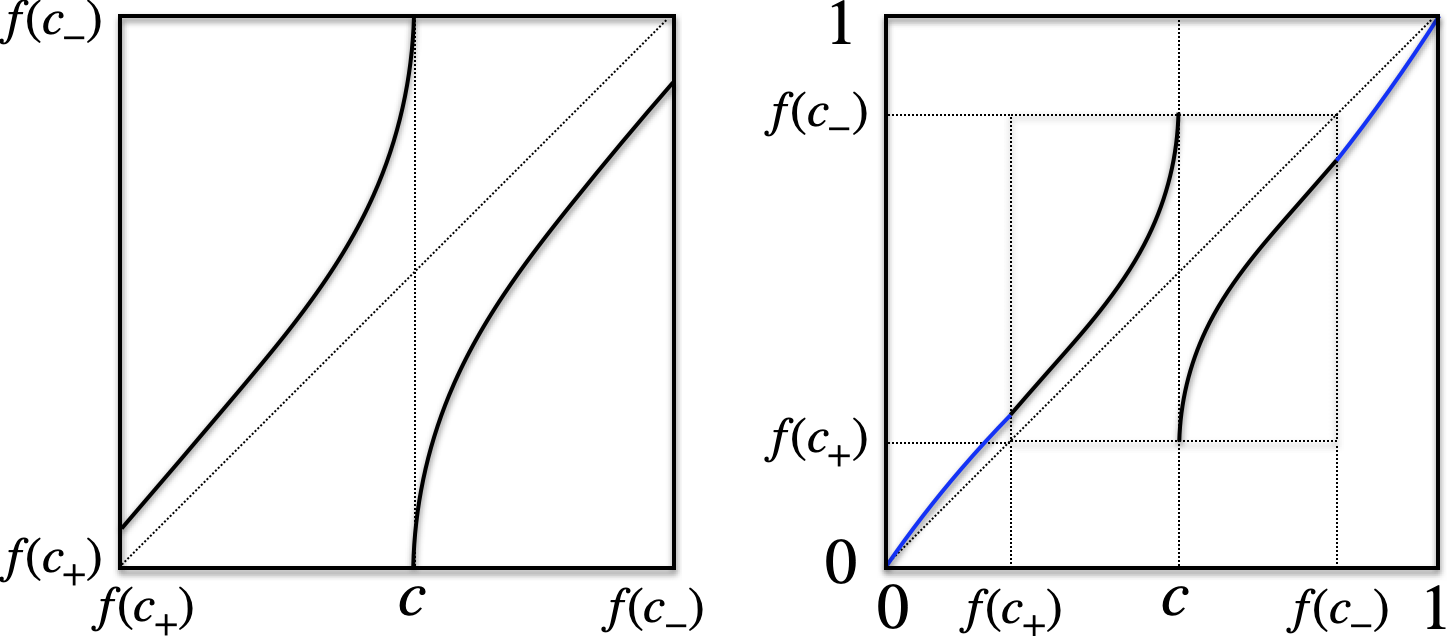}\\
  \caption{\small On the left side we have a Lorenz map from the reduction of Lorenz flow.
  On the right, this map is extended to present two fixed points at the end of the interval.}\label{doismapas}
  \end{center}
\end{figure}

In the study of non-uniform hyperbolicity for system with critical or singular region, the slow recurrence to the critical/singular region plays a crucial  role.
For such systems, it is common to assume a non-degenerated (non-flat) critical/singular region.
For these systems, almost all points, with respect to  invariant probability $\mu$, have slow recurrence to the critical/singular region if and only if they have only finite Lyapunov exponents.
Let $\cm^1(f)$ be the {\bf\em set of all $f$-invariant probabilities}. 

\begin{maintheorem}\label{MainTheoA} Let $f:[0,1]\setminus\{c\}\to[0,1]$ be a non-flat $C^{1+}$ expanding Lorenz map with singularity $c\in(0,1)$.
	Suppose that there exist $t\ge1$  and $r>0$ such that $f^{t}(c_+)=c$  and $f^n(c_-)\notin(c,c+r)$ $\forall n\in\NN$.
	Then there exists an uncountable set $\cm_{_{+\infty}}$ of invariant probabilities $\mu$ such that  
\begin{enumerate}[$($i$)$]
\item $\mu$ is ergodic;
\item $\supp\mu=[f(c_+),f(c_-)]$, i.e., full support on the attractor of $f$;
\item $h_{\mu}(f)>0$, i.e., positive entropy;
\item $\int_{x\in\XX}|\log|x-c||d\mu=+\infty$, i.e., fast recurrence to the singularity;
\item $\lim_n\frac{1}{n}\log (f^n)'(x)=+\infty$ for $\mu$ almost every $x$, i.e., infinite Lyapunov exponent.
\end{enumerate}
Furthermore, $$\sup\{h_{\mu}(f)\,;\,\mu\in\cm_{_{+\infty}}\}=\sup\{h_{\mu}(f)\,;\,\mu\in\cm^1(f)\}.$$
\end{maintheorem}

Given $c,\alpha,\beta\in(0,1)$, let $\cl(c,\alpha,\beta)$ be the set of all expanding Lorenz maps satisfying \eqref{Equationjytrdcvbjk} and consider the metric
$$\dist(f,g)=\sum_{x\in\{c_-,c_+\}}|f(x)-g(x)|+\sum_{j\in\{1,2\}}\|\phi_{j,f}-\phi_{j,g}\|_{C^{1+}}.$$
One can show that $(\cl(c,\alpha,\beta),\dist)$ is a complete metric space.

We say that a Lorenz map $f$ has a {\bf\em periodic singularity} $c$  if $f^{\ell}(c_-)=c=f^{\ell}(c_+)$ for some $\ell\ge1$.
It is well known that $\cl_{\per}(c,\alpha,\beta)$, the set of all $f\in\cl(c,\alpha,\beta)$ with a periodic singularity, is dense in   $\cl(c,\alpha,\beta)$ (see for instance Lemma~\ref{FolkLemma} in Appendix).
Hence, as we can apply Theorem~\ref{MainTheoA} to all $f\in\cl_{\per}(c,\alpha,\beta)$, we get the density of the Lorenz maps satisfying the hypothesis of Theorem~\ref{MainTheoA}.

\begin{Remark}\label{RemarkDense}
The set  $\cd$ of Lorenz maps $f\in\cl(c,\alpha,\beta)$ satisfying the hypothesis of Theorem~\ref{MainTheoA} is dense in  $\cl(c,\alpha,\beta)$.
\end{Remark}

As a consequence of Theorem~\ref{MainTheoA}, residually in the set of all Lorenz maps, the set of measures with arbitrarily large Lyapunov exponents realizes the topological entropy of the map, as one can see in Corollary~\ref{MainCorAA}  below.

\begin{maincorollary}\label{MainCorAA}
The set of all $f\in\cl(c,\alpha,\beta)$ such that $h_{top}(f)=\sup\{h_{\mu}(f)\,;\,\int\log f'd\mu\ge\lambda\}$ for every $\lambda>0$ is a residual subset of $\cl(c,\alpha,\beta)$. 
\end{maincorollary}

In Theorem~\ref{MainTheoB} below, we show that when the orbits of singular values have a condition of slow approximation to the singular region, we obtain that all the ergodic invariant probabilities for the system has slow recurrence to the singularity and has finite Lyapunov exponent.

\begin{maintheorem}\label{MainTheoB}
Let $f$ be a non-flat $C^{1+}$ expanding Lorenz map and $c$ its singularity. If
 \begin{equation}\label{propriedadelimsup}
 \limsup_{n\to\infty}\frac{1}{n}\sum_{j=1}^{n}-\log|f^j(c_{\pm})-c|<+\infty
 \end{equation}
 then there is $\Upsilon >0$ such that $$\int_{x\in[0,1]}|\log|x-c||d\mu\le \Upsilon\;\text{ for every  }\mu\in\cm^1(f).$$
 Moreover, if \eqref{propriedadelimsup} holds then,  every $\mu \in \cm^1(f)$ has slow recurrence to the singularity and $f$ has an upper bound for Lyapunov exponent of the invariant measures.
\end{maintheorem}

A {\bf\em SRB-measure} is an ergodic invariant probability with its {\bf\em basin of attraction} 
$$\mathcal{B}(\mu)=\bigg\{x\in[0,1]\,;\,\lim_{n \rightarrow \infty} \frac{1}{n}\sum_{j=0}^{n-1}\varphi\circ f^{j}(x)=\int\varphi d\mu,\; \forall\,\varphi\in C([0,1])\bigg\}$$
having positive Lebesgue measure.
It is well know that every non-flat $C^{1+}$ expanding Lorenz map has a unique SRB measure.

\begin{maincorollary}\label{MainCorbb}
	If the singular values of a non-fat $C^{1+}$ expanding Lorenz belong to the basin of attraction of the SRB measure then every ergodic invariant probability has slow recurrence to the singularity and an upper bound for its Lyapunov exponent.
\end{maincorollary}

A direct consequence of Corollary~\ref{MainCorAA} and \ref{MainCorbb} is that, generically for the expanding Lorens maps, the statistical behavior of the singular values is not ruled by the SRB measure. 

\begin{maincorollary}
For a generic $f\in\cl(c,\alpha,\beta)$ the singular values are not contained in the basin of attraction of the SRB measure. 
\end{maincorollary}

\section{Slow recurrence and finite Lyapunov exponents}\label{teoremaB}

In this section $f:[-1,1] \setminus  \{0\} \rightarrow [-1,1]$ is a $C^{1+}$ non-flat  expanding Lorenz map.
We  assume $c=0$, which  facilitates the computations, and then it would be enough to make a change of coordinates. 
The {\bf\em pre-orbit} of a point $x\in[0,1]$ is the set $\mathcal{O}^{-}_{f}(x):=\bigcup_{n\ge 0}f^{-n}(x)$.
For a point $x\in[0,1]\setminus \mathcal{O}^{-}_{f}(c)$, denote the forward orbit of $x$ by $\mathcal{O}^{+}_{f}(x)=\{f^j(x);j\ge 0\}$.
And, if $\exists \ell \ge 1$ such that $f^\ell (c_{-})=c$ (we can assume that $\ell$ is the smallest positive integer with this property), we define $\mathcal{O}^{+}_{f}(c_{-})=\{f^j(c_{-}); 1 \le j \le \ell \}$, similarly we defined $\mathcal{O}^{+}_{f}(c_{+})$.
 It is well known that, for every expanding Lorenz map $f$, $[f(c_{+}),f(c_{-})]$ is topologically transitive and it attracts every point of the interval $(0,1)$. 

  \begin{Remark}\label{nonflatderivadaf}
If $x<c$, it follows from the definition of non-flat, equation (\ref{Equationjytrdcvbjk}), that  $$|f'(x)|=\alpha\dfrac{|\phi_{0}'(-x)|}{|\phi_{0}(-x)|^{1-\alpha}}=\dfrac{1}{|x|^{1-\alpha}}\left(\alpha\dfrac{|\phi_{0}'(-x)| |x|^{1-\alpha}}{|\phi_{0}(-x)|^{1-\alpha} }\right)= \dfrac{\varphi_{0}(x)}{|x|^{\alpha_{0}}}                                                                                                                                                                                                                                                                                                                                                                                                                                                                                                                                                                                                                                                                                                                                                                                                                                                                                                                                                                                                                                                                                                                                                                       , $$
 where $\alpha_{0}=1-\alpha$ and $\varphi_{0}$ is a Hölder homeomorphism. For $x>0$ it follows the same way.
  \end{Remark} 

Hence, there are Hölder homeomorphisms $ \varphi_0:[-1,0] \rightarrow [f'(-1),d_{0}]$ and $\varphi_1:[0,1] \rightarrow [d_{1},f'(1)]$ such that 
 \begin{equation}\label{nonflatderf}
f'(x) =\left\{ \begin{matrix}\varphi_{0}(x)/|x|^{\alpha_{0}} &
      \mbox{if }x <0
      \\ \varphi_{1}(x)/|x|^{\alpha_{1}}  & \mbox{if }x > 0 \end{matrix} \, , \right. 
\end{equation}   
     where $\alpha_{0}, \alpha_{1} \in (0,1) $ and $ d_{0}=\lim_{x\uparrow 0}f'(x)|x|^{\alpha_{0}} , d_{1}=\lim_{x\downarrow 0}f'(x)|x|^{\alpha_{1}} $.

  Thus, from (\ref{nonflatderf}) we can take $a>1$ and  $0 < \alpha \le \beta < 1$ so that 
\begin{equation}\label{nonflatderflimitada}
 \frac{1}{a}|x|^{-\alpha}\le f'(x)\le a |x|^{-\beta}.
\end{equation}   
  
As $f'>0$, we have $\inf \varphi_{i}>0$, $i=0,1$.
So, $\log |\varphi_{i}|$ is Hölder.
Let $C,t>0$ be so that
\begin{equation*}
|\log |\varphi_{i}|(x) - \log |\varphi_{i}|(y)|\le C |x-y|^t,\  \forall i,
\end{equation*}
and so
\begin{equation}\label{lip}
\frac{|\varphi_{i}(x)|}{|\varphi_{i}(y)|}\le e^{C|x-y|^t}.
\end{equation}

\begin{Remark} For all $x \in \RR \setminus \{0\}, y \in \RR$ and $r>0$, we have
  \begin{eqnarray}\label{obs}
  \left|\dfrac{y}{x}\right|^{r}&\le & 
  \left( \left|\dfrac{y-x}{x}\right| + 1\right) ^{r}
  \le  \left( e^{  \frac{|y-x|}{|x|} }\right)^{r}=  e^{ r |y-x|/|x| }.
   \end{eqnarray}
\end{Remark}

\begin{Remark}
If $xy > 0$ then
 \begin{eqnarray}\label{obs2}
  |x - y |< \frac{|x|}{2} &\Leftrightarrow & \left| 1 - \dfrac{y}{x}\right|< \frac{1}{2}    \Leftrightarrow \frac{1}{2}< \frac{|y|}{|x|} < \frac{3}{2} \Rightarrow \frac{1}{|y|}< \frac{2}{|x|}.
 \end{eqnarray}
 \end{Remark}

From \eqref{lip}, \eqref{obs} and \eqref{obs2} we obtain a certain control of the derivatives of $f$ in each branch of its domain. For this, consider $\alpha= \alpha_{i}, \varphi = \varphi_{i}$ e $d= d_{i}$ for each case when $x < 0$ or $x > 0$. Thus, $ f'(x)= \varphi (x)/|x|^{\alpha}$ and $\varphi (0) = d >0$.

\begin{Lemma}\label{lemacontroledistn} There is $\gamma >0$ such that for every $n\ge 1$ and every $x, y$ so that $|f^{j}(x)- f^{j}(y)|\le \frac{|f^{j}(x)|}{2}$ and $f^{j}(x)f^{j}(y)>0$ for every $0 \le j <n $, we have
\begin{equation}\label{distn}
 \dfrac{|(f^{n})'(y)|}{|(f^{n})'(x)|}\le  e^{\gamma \sum_{j=0}^{n-1}\big(\frac{| f^{j}(x)- f^{j}(y)|}{|f^{j}(x)|}\big)^t}.
\end{equation}
\end{Lemma}

\begin{proof}
Let $ x, y \in [-1,1] \setminus \{0\}$ such that $xy > 0$ and $| x-y | \le \frac{|x|}{2}$.
Then
\begin{eqnarray}\label{dist}
 \dfrac{|f'(y)|}{|f'(x)|}= \dfrac{|\varphi (y)|}{|\varphi (x)|}  \dfrac{|x|^{\alpha}}{|y|^{\alpha}}  & \le & e^{C|y- x|^t} \left| \dfrac{x}{y}\right|^{\alpha} \le  e^{C\big(\frac{|y- x|}{|x|}\big)^t}e^{2\alpha\frac{| x - y|}{|x|}} \nonumber \\ 
 & \le & e^{(C + 2\alpha)\big(\frac{| x - y|}{|x|}\big)^t} \le e^{\gamma\big(\frac{|x - y|}{|x|}\big)^t},
\end{eqnarray}
where we will take $\gamma = C + 2 \max \{ \alpha_{0}, \alpha_{1} \}$.

Consequently, using \eqref{dist}, for every $n \ge 1$ and every $x, y$ such that $| f^{j}(x)- f^{j}(y) | \le \frac{|f^{j}(x)|}{2}$ and $f^{j}(x)f^{j}(y) > 0$ for every $0 \le j < n $, we have
\begin{equation*}
  \frac{|(f^{n})'(y)|}{|(f^{n})'(x)|}= \prod_{j=0}^{n-1} \frac{|f'(f^{j}(y))|}{|f'(f^{j}(x))|}\le \prod_{j=0}^{n-1} e^{\gamma\big(\frac{| f^{j}(x)- f^{j}(y)|}{|f^{j}(x)|}\big)^t} = e^{\gamma \sum_{j=0}^{n-1}\big(\frac{| f^{j}(x)- f^{j}(y)|}{|f^{j}(x)|}\big)^t}.
\end{equation*}
\end{proof}

\begin{Definition}\label{bound period}
Fix $0 < \delta \le \frac{1}{2}$ and for any $p\ne 0$, we will call $\delta$-bound period (or simply bound period) of $p$ with $0$ the number $$m(p):= -1 + \min\{j>0\,;\,|f^{j}(p)-f^{j}(0_{p})|\ge\delta|f^{j}(0_{p})|\},$$ 
where
 $$ 0_{p}=
 \begin{cases}
 0_{-} & \text{ if }p<0\\
 0_{+} & \text{ if }p>0
 \end{cases}.$$
\end{Definition}

This means that for every $1 \le j \le m(p)$, we have, either $f^{j}(p)$ and $f^{j}(0_{p}) < 0$ or $f^{j}(p)$ and $f^{j}(0_{p}) > 0$, i.e.,
\begin{equation*}
  \frac{| f^{j}(p)- f^{j}(0_{p})|}{| f^{j}(0_{p})|}< \delta \ \mbox{   and  } \ f^{j}(p)f^{j}(0_{p})>0. 
\end{equation*}
Consequently,
\begin{equation*}
   \sum_{i=1}^{j}\bigg(\frac{| f^{i}(p)- f^{i}(0_{p})|}{| f^{I}(0_{p})|}\bigg)^t<\delta^t j  \  \mbox{for every}\  1 \le j \le m(p),
\end{equation*}
and then of \eqref{distn}, for $x=f(0_{p})$ and $ y=f(p)$ we have
\begin{equation}
   \frac{|(f^{j})'(f(p))|}{|(f^{j})'(f(0_{p}))|} \le e^{\gamma \sum_{i=0}^{j-1}\big(\frac{| f^{i}(f(p))- f^{i}(f(0_{p}))|}{|f^{I}(f(0_{p}))|}\big)^t} < e^{\gamma \delta^t j},   
\end{equation}
for every  $ 1 \le j \le m(p)$. 

Also, as $f^{m(p)+1}\mid  _{(0,p)}$ if $p>0$ or $f^{m(p)+1}\mid  _{(p,0)}$ if $p<0$ is diffeomorphism, for every  $ y \in \{ tp + (1 - t)(0_{p}); t \in [0,1] \}$ and $ 1 \le j \le m(p)$ we have
\begin{equation}\label{contder}
  |(f^{j})'(f(y))| < e^{\gamma \delta^t j}|(f^{j})'(f(0_{p}))|.
\end{equation}

\begin{Remark}\label{contprof}
It follows from \eqref{nonflatderflimitada} that, if  
 $M=\sup_{n\in\NN}\frac{ 1}{ n} \sum_{j=1}^{n}|\log |f^{j}(0_{\pm})||<+\infty$
 then, for each $n\in\NN$, 
 \begin{enumerate}[$($i$)$]
 \item $\prod_{j=1}^{n}|f^{j}(0_{\pm})| \ge  e^{-M n }$, in  particular, $|f^{j}(0_{\pm})|\ge e^{-M n }$ for all  $1 \le j \le n$;
 \item $|(f^{j})'(f(0_{\pm}))|\le e^{b n}$ for all $1\le j\le n$, where $b:= \log a + \beta M$. 
 \end{enumerate}
\end{Remark}

The next result provides us with a relation between the \textit{depth} at which a point is and the \textit{bound period} in which its orbit is linked to the singular orbit. Intuitively, it means that when a point is close to the singularity, we have an estimate for the bound period with the orbit of the singular point.
Thus, assuming that
 \begin{equation}\label{propsup}
   M:=\sup_{n\in\NN}\frac{ 1}{ n} \sum_{j=1}^{n}|\log |f^{j}(0_{\pm})||<+\infty,
\end{equation} we have the following consequences.

\begin{Corollary}\label{estimaboundperiod}
Let $0 < A \le \delta  \frac{ 1- \beta}{ a^{2}}$ and $B \ge 2 \gamma \delta^t + 2M +b$.
If $|p|^{1 - \beta} \le Ae^{-Bn}$, with $n\ge 1$, then $m(p) \ge n$.
\end{Corollary}
\begin{proof}
Initially note that, by \eqref{nonflatderflimitada},
\begin{equation*}
| f(p)- f(0_{p})|=  \left| \int_{0_{p}}^{p} f'(x)dx \right|\le  \int_{0_{p}}^{p}| f'(x)|dx \le a \int_{0}^{p}| x|^{- \beta}dx = \dfrac{a}{1- \beta} | p|^{1- \beta}
\end{equation*}
By the {\em mean value theorem} and the definition of $m(p)$, we have that 
 \begin{equation*}
| f^{j}(p)- f^{j}(0_{p})|= |(f^{j-1})'(y)| | f(p)- f(0_{p})|,  
\end{equation*}
 for each $1< j \le m(p)+1$, where $y \in (f(0_{p}), f(p))$.
 Then, it follows from \eqref{contder} that
\begin{eqnarray*}
| f^{j}(p)- f^{j}(0_{p})| & \le & e^{\gamma \delta^t (j-1)}\prod_{i=1}^{j-1}|f'(f^{i}(0_{p}))|\  | f(p)- f(0_{p})| \\ & \le &  e^{\gamma \delta^t (j-1)}\prod_{i=1}^{j-1}|f'(f^{i}(0_{p}))|\  \frac{a }{ 1- \beta} |p|^{1 - \beta} \\
   & \le & e^{\gamma \delta^t (j-1)}\prod_{i=1}^{j-1}|f'(f^{i}(0_{p}))|\  \frac{a }{ 1- \beta}A e^{-Bn}\\ & \le & e^{\gamma \delta^t (j-1)}\prod_{i=1}^{j-1}|f'(f^{i}(0_{p}))|\  \frac{\delta }{ a}e^{-Bn}, 
\end{eqnarray*}
for every $ j \le m(p)+1 $.
Suppose by contradiction that $m(p)< n$.
In this case, using item $(ii)$ of the Remark~\ref{contprof}, we have 
 \begin{equation}\label{contint}
 | f^{m(p)}(p)- f^{m(p)}(0_{p})| \le \frac{\delta }{ a}e^{\gamma \delta^t (m(p))}e^{-Bn} e^{bn} <  \frac{\delta }{ a}e^{(\gamma \delta^t -B + b)n}.
 \end{equation}
As $|f'(y)|/|f'(f^{m(p)}(0_p))|< e^{\gamma \delta^t}$ for each $y \in \{ tf^{m(p)}(p) + (1-t)f^{m(p)}(0_p); t \in [0,1] \}$, by  (\ref{contder}), we obtain that
\begin{eqnarray*}
\frac{| f^{m(p)+1}(p)- f^{m(p)+1}(0_{p})|}{ | f^{m(p)+1}(0_{p})|  } & = & \left|\int_{f^{m(p)}(0_{p})}^{f^{m(p)}(p)}\frac{|f'(y)|}{|f'(f^{m(p)}(0_p))|}dy \right|\frac{|f'(f^{m(p)}(0_p))|}{| f^{m(p)+1}(0_{p})|} \\
   & \le & e^{\gamma \delta^t}  | f^{m(p)}(p)- f^{m(p)}(0_{p})|\frac{|f'(f^{m(p)}(0_p))|}{| f^{m(p)+1}(0_{p})|} \\
   & \underset{\eqref{contint}  }{<} & e^{\gamma \delta^t}   \dfrac{\delta }{ a}e^{(\gamma \delta -B + b)n} a \dfrac{|f^{m(p)}(0_p)|^{-\beta}}{| f^{m(p)+1}(0_{p})|} \\
 &  \underset{Remark~\ref{contprof} (i) }{\le}  &  \delta e^{\gamma \delta^t}   e^{(\gamma \delta^t -B + b)n} e^{\beta M n }  e^{M n }   \\
    & < &  \delta e^{(2 \gamma \delta^t + b + 2M - B) n } < \delta,
    \end{eqnarray*} contradicting  the definition of $m(p)$.
    Therefore $m(p) \ge n$ and we conclude the proof.
\end{proof}

\begin{Corollary}\label{contsomaatem} 
If $n\ge 1$ is such that $ Ae^{-B(n+1)}<  |p|^{1 - \beta}  \le Ae^{-Bn}$ then
\begin{equation}\label{contsomaatemmenorgamma}
\sum_{j=0}^{m(p)}|\log|f^j(p)||< \Gamma m(p)
\end{equation}
where $$\Gamma=\dfrac{|\log A|}{1 - \beta}+ 2\dfrac{B}{1 - \beta}+\log \left(\dfrac{1}{1-\delta}\right) + M.$$
\end{Corollary}

\begin{proof}
As 
\begin{equation*}
 |f^j(p)|   >    |f^{j}(0_{p})|- |f^j(p)-f^{j}(0_{p})|  >  (1-\delta) |f^{j}(0_{p})|,\   \forall 1 \le j \le m(p), 
\end{equation*}
 it follows that  
\begin{equation}\label{logfp}
|\log |f^j(p)||   <  \log\left(\dfrac{1}{1-\delta}\right)+ |\log|f^{j}(0_{p})||,\   \forall 1 \le j \le m(p). 
\end{equation}

We know from the previous Corollary that $m(p)\ge n$.
Now from $ Ae^{-B(n+1)}<  |p|^{1 - \beta}$, we have
\begin{eqnarray}\label{logp}
  \log A +(-B)(n+1)  <  (1 - \beta)\log |p|  & \Rightarrow & \left| \dfrac{-\log A}{(1 - \beta)} + \dfrac{B(n+1)}{(1 - \beta)}\right|  >  \left| -\log |p|\right|\nonumber \\
  & \Rightarrow &\dfrac{|\log A|}{(1 - \beta)} + \dfrac{B(n+1)}{(1 - \beta)}  >  |\log |p||.
\end{eqnarray}
As a consequence of \eqref{logfp} and \eqref{logp} it follows that
\begin{eqnarray*}
 \sum_{j=0}^{m(p)}|\log|f^j(p)|| & \le &  |\log |p||+  \sum_{j=1}^{m(p)}|\log|f^j(p)|| \\
 &  <   & |\log |p||+  \sum_{j=1}^{m(p)}\left[\log\left(\dfrac{1}{1-\delta}\right)+ |\log|f^{j}(0_{p})||\right] \\
 &   <  &  \dfrac{|\log A|}{(1 - \beta)} + \dfrac{B(n+1)}{(1 - \beta)} + m(p)\log\left(\dfrac{1}{1-\delta}\right)+ \sum_{j=1}^{m(p)} |\log|f^{j}(0_{p})|| \\
 &  \underset{\eqref{propsup}  }{<}    &   \dfrac{|\log A|}{(1 - \beta)} + (n+1)\dfrac{B}{(1 - \beta)} + m(p).\log\left(\dfrac{1}{1-\delta}\right)+ Mm(p) \\
 &   <  &  \dfrac{|\log A|}{(1 - \beta)} + 2m(p)\dfrac{B}{(1 - \beta)} + m(p)\log\left(\dfrac{1}{1-\delta}\right)+ Mm(p) \\
 &  <   &   \left[\dfrac{|\log A|}{(1 - \beta)} + 2\dfrac{B}{(1 - \beta)} + \log\left(\dfrac{1}{1-\delta}\right)+ M \right] m(p).
\end{eqnarray*}

Therefore,
\begin{equation*}
 \sum_{j=0}^{m(p)}|\log|f^j(p)||< \Gamma m(p).
\end{equation*}
\end{proof}

The goal here is to obtain a control for the recurrence of points to the singularity.
Note that by this last Corollary, whenever a point is in the neighborhood of the singularity $c = 0$ we obtain a control to its orbit until a certain time, the bound period.
And so, we are going to take a $V_{r}$ neighborhood of $0$ for that purpose, as in the Fig~\ref{vizinhancaVr2}.

\begin{figure}
\begin{center}
\includegraphics[scale=.55]{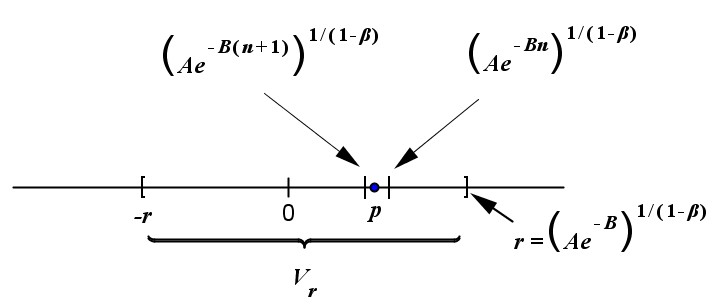}
\caption{\small Neighborhood $V_{r}$ of $0$.}\label{vizinhancaVr2}
\end{center}
\end{figure}

\subsection{Proof of Theorem~\ref{MainTheoB} and Corollary~\ref{MainCorbb}}

\begin{proof}[Proof of the Theorem~\ref{MainTheoB}]

Remember that we are considering $c = 0$ and suppose that $$\limsup_{n}\frac{1}{n}\sum_{j=1}^{n}-\log|f^j(0_{\pm})|<+\infty.$$
That is,
   $0<M:=\sup_{n}\frac{ 1}{ n} \sum_{j=1}^{n}|\log |f^{j}(0_{\pm})||<+\infty$
and so, by Remark~\ref{contprof}, we are allowed to use Corollary~\ref{estimaboundperiod} and Corollary~\ref{contsomaatem}.

Let $r=(Ae^{-B})^{\frac{1}{1- \beta}}$ and $U= \{x \in [-1,1] \setminus \mathcal{O}^{-}_{f}( 0) ; 0 \in \omega_{f}(x)\}  $ and let $R:U \to \{0, 1, 2, \cdots \}$  be the map that gives the first entry time to $V_{r}=[-r,r]$, i.e., $R(x)=\min\{ 0\le j; |f^{j}(x)| \le r \}$.

Given a point $p \in U$, let $\ell_{0}=0$, $r_{0}= R(f^{\ell_{0}}(p))$ and $n_{0}= m(f^{\ell_{0}+ r_{0}}(p))$.
Note that even before the entry time in $V_{r}$, we have the following estimate for the orbit of the point $p$
\begin{equation*}
 \sum_{j=0}^{r_{0}-1}|\log|f^j(p)|| \le r_{0} \log\dfrac{1}{r}
\end{equation*}
and for $p_{0}=f^{ r_{0}}(p)=f^{\ell_{0}+ r_{0}}(p)$, we know from \eqref{contsomaatemmenorgamma}, that until the bound period $n_{0}$ we have 
\begin{equation*}
 \sum_{j=0}^{n_{0}}|\log|f^j(p_{0})|| =   \sum_{j=r_{0}}^{r_{0}+n_{0}}|\log|f^j(p)||< \Gamma n_{0}.
\end{equation*}
Then
\begin{eqnarray*}
 \sum_{j=0}^{\ell_{0}+r_{0}+n_{0}}|\log|f^j(p)|| & = &   \sum_{j=0}^{r_{0}-1}|\log|f^j(p)|| +  \sum_{j=r_{0}}^{\ell_{0}+r_{0}+n_{0}}|\log|f^j(p)|| \\
& \le &  r_{0} \log (1/r)+  \Gamma n_{0}  \\
& \le & ( \ell_{0}+r_{0}+n_{0}) \left(  \log (1/r)+  \Gamma \right)\\
& = & ( \ell_{1}) \left(  \log (1/r)+  \Gamma \right).
\end{eqnarray*}
Inductively, let $\ell_{j+1}=\ell_{j}+r_{j}+n_{j}$, $r_{j+1}=R(f^{\ell_{j+1}}(p))$ and $n_{j+1}= m(f^{\ell_{j+1}+ r_{j+1}}(p))$.
Besides that, we have
\begin{eqnarray*}
 \sum_{j=0}^{\ell_{k}}|\log|f^j(p)|| & = &  \sum_{j=0}^{k}  \sum_{i=\ell_{j}}^{\ell_{j+1}-1}|\log|f^i(p)||  \\
 & = &  \sum_{j=0}^{k} \left( \sum_{i=0}^{r_{j}-1}|\log|f^i(f^{\ell_{j}}(p)||  + \sum_{i=0}^{n_{j}}|\log|f^i(f^{\ell_{j}+r_{j}}(p)||\right)  \\
& \le &  \sum_{j=0}^{k} \left(\sum_{i=0}^{r_{j}-1}\log(1/r)+  n_{j}\Gamma \right)\\
& = &   \sum_{j=0}^{k}\left( r_{j}\log(1/r)+  n_{j}\Gamma \right)\\
& \le & (\ell_{k} +1)\left( \log (1/r)+ \Gamma \right).
\end{eqnarray*}
Therefore,
\begin{equation}\label{liminfcontrolado}
\lim\inf \dfrac{1}{n} \sum_{j=0}^{n-1}|\log|f^j(p)||< \log (1/r)+ \Gamma =\Upsilon,
\end{equation}
for every $p\in [-1,1] \setminus \mathcal{O}^{-}_{f}( 0)$ ( It is clear that the equation~\eqref{liminfcontrolado} also applies to points $p$ such that $0 \notin \omega_{f}(p)$). Thus, by ergodic theorem of Birkhoff, for every ergodic $f$-invariant probability $\mu$, we have
\begin{equation}\label{recurrencecontrolado}
\int_{x\in[-1,1]} |\log |x||d\mu < \Upsilon.
\end{equation}
This means that the logarithm of the distance to the singular point $c=0$ is $\mu$ integrable. Consequently,
\begin{equation*}
\int\log\dist_{\epsilon^{-n}}(x,c)d\mu     \underset{ \delta-trunc  }{=}   \int_{[x; \log\dist(x,0)<-n]} \log |x|d\mu \to 0
\end{equation*}
when $n \to \infty$. And so, by Birkhoff, $\mu$ satisfies the condition of slow recurrence. 

Now, as  $f$ is non-flat, i.e., $ \frac{1}{a}|x|^{-\alpha}\le f'(x)\le a |x|^{-\beta}$, we get for every $\mu \in \cm^1(f)$,
\begin{equation}\label{explyapandrecslow}
 \log \left(\frac{1}{a}\right) + \alpha\int_{x\in[-1,1]} |\log |x||d\mu \le \int_{x\in[-1,1]}\log f'd\mu \le \log a + \beta \int_{x\in[-1,1]}|\log|x||d\mu.
\end{equation} 
Using \eqref{recurrencecontrolado}, it follows
\begin{equation*}\label{explyapfinito}
 0< \int\log f'd\mu < \log a + \beta \Upsilon. 
\end{equation*} 
Therefore, again by Birkhoff's theorem, for every ergodic $f$-invariant probability $\mu$, we obtain
\begin{equation*}\label{explyapfinito}
    \lim_{n\to\infty}\dfrac{1}{n}\log|(f^{n})'(x)| = \int\log f'd\mu \ne \pm \infty,
\end{equation*} 
that is, $\mu$ has  finite Lyapunov exponent. This finishes the proof of the theorem.
\end{proof}

\begin{proof}[Proof of Corollary~\ref{MainCorbb}]
In fact, as $\mu \ll m$, it follows from  the Ledrappier-Young entropy formula and the variational principle that
 $$\int\log f' d\mu = h_{\mu}(f) \le h_{top}(f) < \infty.$$

And as $f$ is non-flat, we have by inequality~\label{explyapandrecslow} that $\int_{x\in[-1,1]} |\log |x||d\mu < \infty$.

So,  if $f(0_{\pm}) \in B(\mu)$ we have
\begin{equation*}
\lim \dfrac{1}{n} \sum_{j=0}^{n-1}|\log|f^j( f(0_{\pm}) )||= \int_{x\in[-1,1]} |\log |x||d\mu < \infty
\end{equation*}
and therefore $f$ satisfies the condition~\eqref{propriedadelimsup}.
\end{proof}



\section{Measures with  infinite Lyapunov exponent and positive entropy }\label{teoremaA}

Let $f:[0,1]\setminus\{c\}\to[0,1]$, $0<c<1$, is a $C^1$ non-flat  expanding Lorenz map, and $\cm^1(f)$ the set of all $f$-invariant Borel probabilities. A $f$-induced map defined on $A\subset [0,1]\setminus\{c\}$ with an induced time $R:A\to\NN:=\{1,2,3,\cdots\}$ is the map $F:A\to [0,1]$ defined by $F(x)=f^{R(x)}(x)$.
An induced time $R:A\to\NN$ is called {\bf\em exact} if  $R(x)=R(f^j(x))+j$ whenever $x\in A$ and $f^j(x)\in A$ for every $0\le j<R(x)$. Note that the {\bf \em first entry times}, is a particular case of an exact time. An induced map $F$ is called {\bf\em orbit-coherent} if  $\co_f^+(x)\cap\co_f^+(y)\ne\emptyset\Longleftrightarrow\co_F^+(x)\cap\co_F^+(y)\ne\emptyset
$ for every $x,y\in\bigcap_{j\ge0}F^{-j}([0,1])$. 

\color{black}

A {\bf\em full induced Markov map} is a triple $(F,B,\cp)$ where $B\subset[0,1]\setminus\{c\}$ is a connected open set, $\cp$ is countable collection of disjoints open subsets of $B$ and  $F:A:=\bigcup_{P\in\cp}P\to B$ is an $f$ induced map satisfying: 
\begin{enumerate}
	\item for each $P\in\cp$, $F|_P$ is a  diffeomorphism between $P$ and $B$ and it can be extended to a homeomorphism sending $\overline{P}$ onto $\overline{B}$;
	\item $\lim_n\diameter(\cp_n(x))=0$ for every $x\in\bigcap_{n\ge1}F^{-n}(B)$, where $\cp_n=\bigvee_{j=0}^{n-1}F^{-j}(\cp)$ and $\cp_n(x)$ is the element of $\cp_n$ containing $x$.
\end{enumerate}

If $F$ is a $f$-induced Markov map with an exact induced time,  then $F$ is orbit-coherent (see Lemma~2.6 of \cite{Pi20}).
Furthermore, when $F$ is orbit-coherent, an ergodic $f$-invariant $\mu$ has at most one $F$-lift and this lift must be a $F$-ergodic probability (Theorem~A of \cite{Pi20}).

A {\bf\em mass distribution} on $\cp$ is a map $m:\cp\to[0,1]$ such that $\sum_{P\in\cp}m(P)=1$.
The {\bf\em $F$-invariant probability $\mu$ generated by the mass distribution} $m$ is the ergodic $F$ invariant probability $\mu$ defined by $$\mu(P_1\cap F^{-1}(P_2)\cap\cdots\cap F^{n-1}(P_n))=m(P_1) m(P_2)\cdots m(P_n),$$ where $P_1\cap F^{-1}(P_2)\cap\cdots\cap F^{n-1}(P_n)\in\bigvee_{j=0}^{n-1}F^{-j}(\cp)$.

\begin{Lemma}[Lemma C.1 of \cite{Pi20}]\label{LemmaExacRtoRR}
Let $f$ be a measure preserving automorphism on a probability space $(\XX,\mathfrak{A},\mu)$ and $F:A\subset\XX\to\XX$ a measurable induced map with induced time $R:A\to\NN$.
Suppose that $\mu$ is $f$ ergodic and that $\nu$ is a $F$-lift of $\mu$.
If $R$ is exact and $\mu(A)=1$ then 
$$\frac{1}{2}\int Rd\nu \int R d\mu\le\int (R)^2 d\nu \le 2\int Rd\nu\int R d\mu.$$
\end{Lemma}

The Lemma~\ref{FolkloreResultB} below is a well-known result  (see, for instance, Lemma~4.2 of \cite{Pi20} for a short proof of it).

\begin{Lemma}[Folklore result]\label{FolkloreResultB}
If $F:A\to\XX$ is a measurable $f$ induced map with induced time $R$ and $\nu$ is a $F$ invariant probability then $$\mu:=\sum_{n\ge1}\sum_{j=0}^{n-1}f^j_*(\nu|_{\{R=n\}})=\sum_{j\ge0}f^j_*(\nu|_{\{R>j\}})
$$ is a $f$ invariant measure with $\mu(\XX)=\int R d\nu$.
\end{Lemma}

\subsection{Proof of the Theorem~\ref{MainTheoA} and Corollary~\ref{MainCorAA}}

The proof of Theorem~\ref{MainTheoA} follows the following script.
Because of the geometric assumption on the singular values, the first return map $F=f^R$ to a lateral neighborhood of the singularity $c$ is a full induced Markov map.
Moreover, $c$ belongs to the boundary a connected component $P$ of the domain of $F$.
We use that $c\in\partial P$ to build a new induced map $F_c=F^{R_c}$ with good properties.
In particular, every $f$-invariant ergodic probability $\mu$ with $\supp\mu=[f(c_+),f(c_-)]$ can be $F_c$-lifted.  
Given an ergodic $f$ invariant probability $\mu$, with $\supp\mu=[f(c_+),f(c_-)]$, one can use mass distribution to obtain a $F_c$-invariant Bernoulli measure $\nu$ arbitrarily close to the $F_c$-lift of $\mu$ (in particular, its entropy).
Furthermore, we can show that $\int R_c d\nu$ and $\int R d\mu_1<+\infty$, but $\int R_c d\nu_2=+\infty$, where $\nu_1$ is the $F$-invariant probability having $\nu$ as its $F_c$-lift and, similarly, $\nu_2$ is the $f$-invariant probability having $\nu_1$ as its $F$-lift.
As we can compere $\log|x-c|$ and $\log f'$ with $R_c$, we can construct the measures $\cm_{_{+\infty}}$ of Theorem~\ref{MainTheoA}.

\begin{proof}[Proof of the Theorem~\ref{MainTheoA}] 
Let $p\in(c,c+r)$ be a periodic point such that $p=\min(\co_f^+(p)\cap(c,1)) <  \min(\co_f^+(c_+)\cap(c,1))$.
Let $J=(c,p)$ and note that $\co_f^+(\partial J)\cap J=\emptyset$.
Let $A=\{x\in J\,;\,\co_f^+(f(x))\cap J\ne\emptyset\}$, $R:A\to\NN$ be the first return time to $J$  and $F:A\to J$ the first return map, i.e., $F(x)=f^{R(x)}$.

Let $\cp$ be the set of connected components of $A$.

\begin{Claim}\label{Claim809098787}
$F(P)=(c,p)$ for every $P\in\cp$.	
\end{Claim}
\begin{proof}[Proof of the claim]
Indeed, given $x\in A$, let $I=(a,b)\ni x$ be the maximal open interval such that $f^{R(x)}|_{I}$ is a homeomorphism and $f^{R(x)}(I)\subset J$.
If $f^{R(x)}(I)\ne J$, then $f^{R(x)}(a_+)\in J$ or $f^{R(x)}(b_-)\in J$. Suppose for instance that $f^{R(x)}(a_+)\in J$.
If $f^j(a_+)=c$ for some $1\le j<R(x)$ then $f^{R(x)-j}(c_+)=f^{R(x)}(a_+)\in J$ which is a contradiction with the definition of $p$.
So, $f^j(a_+)\ne c$ for every $0\le j<R(x)$ and in this case, there is $\delta>0$ such that $f^{R(x)}|_{B_{\delta}(a)}$ is a homeomorphism and $f^{R(x)}(B_{\delta}(a))\subset J$, but this implies that $f^{R(x)}|_{(a-\delta,b)}$ is a homeomorphism with $f^{R(x)}((a-\delta,b))\subset J$, contradicting the definition of $I(x)$.
Now suppose that $f^{R(x)}(b_-)\in J$. If $f^j(b_-)=c$ for some $1\le j<R(x)$ then $f^{R(x)-j}(c_-)=f^{R(x)}(b_-)\in J$, contradicting that $\co_f^+(c_-)\cap (c,c+r)=\emptyset$. So, there is $\delta>0$ such that $f^{R(x)}|_{B_{\delta}(b)}$ is a homeomorphism and $f^{R(x)}(B_{\delta}(b))\subset J$, but this implies that $f^{R(x)}|_{(a,b+\delta)}$ is a homeomorphism with $f^{R(x)}((a,b+\delta))\subset J$, contradicting the definition of $I(x)$. Hence, we must have $f^{R(x)}(I(x))=(c,p)$, proving  the claim.
\end{proof}

\begin{Claim}\label{Claim0100020010} $\exists q>c$ so that $(c,q)\in\cp$ and $R((c,q))=t_0:=\min\{j\ge1\,;\,f^j(c_+)=c\}$.
\end{Claim}
\begin{proof}[Proof of the claim]
it follows from the definition of $t_0$ that exists $\delta>0$ such that $f^{t_0}|_{(c,c+\delta)}$ is a homeomorphism and $f^{t_0}((c,c+\delta))=(c,f^{t_0}(c+\delta))$.
Moreover, $f^{j}((c,c+\delta))\cap J=\emptyset$ for every $1\le j<t_0$. Hence, $(c,c+\delta)\subset A$.
Hence, the connected component of $A$ containing $(c,c+\delta)$ must be $(c,q)$ for some $c+\delta\le q\le p$ and $R((c,q))=t_0$.
\end{proof}

Let $P_0:=(c,q)\in\cp$  be given by Claim~\ref{Claim0100020010}. 
Let $\cp_0=\{J\}$ and, for $n\ge1$, set $\cp_n=\bigvee_{j=0}^{n-1}F^{-j}(\cp)$.
Let $\cp_n(c):=(F_{P_0})^{-n}(J)$, that is, $\cp_n(c)$ is the element of the partition $\cp_{n}$ of $J$ containing $(c,c+\delta_n)$ for some $\delta_n>0$.

Let $R_c:A\to\NN$ be an $F$-induced time given by
$$R_c(x)=\
\begin{cases}
1 & \text{ if }x\in A\setminus P_0\\
(n+1) & \text{ if }x\in \cp_n(c)\setminus\cp_{n+1}(c)\text{ for }n\in\NN.
\end{cases}$$

Let $\widetilde{A}= \bigcup_{n\ge 0}\left(\cp_n(c)\setminus\cp_{n+1}(c)\cap (\cap_{j=0}^n F^{-j}(A))\right) $ and define $F_c:\widetilde{ A}\to B$ by $$F_c(x)=F^{R_c(x)}(x)$$ and note that $F_c$ is orbit-coherent with respect to $F$.
Let $\cp_{c,0}=\cp_0\cap(J\setminus\cp_1(c)):=\{P\in\cp\,;\,P\subset J\setminus\cp_1(c)\}$ and $\cp_{c,n}=\cp_n\cap(\cp_n(c)\setminus\cp_{n+1}(c))$ for $n\in\NN$.
Let $A^*=\{x\in J\,;\,\co_F^+(x)\cap(q,p)\cap A\ne\emptyset\}$.
Taking $$\cp^*=\bigcup_{n\ge0}\cp_{c,n},$$ we have that $\cp^*$ is a partition of $A^*$ and that $F_c(P)=J$ for every $P\in\cp^*$.

Let $\mu$ be a $f$ invariant ergodic probability with $\supp\mu=[f(c_+),f(c_-)]$. Note that $\mu_0=\frac{1}{\mu((c,p))}\mu|_{(c,p)}$ is a $F$ invariant ergodic probability and
\begin{equation}\label{Equation010010101112}
  \lim_n\frac{1}{n}\#\{0\le j<n\,;\,F^j(x)\in \co_{F_c}^+(x)\}=\mu_0((q,p))>0
\end{equation}
for $\mu_0$ almost every $x$.
As $F_c$ is orbit-coherent $F$-induced map and \eqref{Equation010010101112} holds, it follows from Theorem~A that $\mu_0$ has a unique $F_c$-lift $\mu_c$ and this $\mu_c$ is $F_c$-ergodic.

Given $\alpha\in(0,1)$ and $\ell\in\NN$, consider the mass distribution

$$m_{\ell}(P)=
\begin{cases}
\hspace{2cm}\mu_c(P) & \text{ if }P\in\bigcup_{j=0}^\ell\cp_{c,j}\vspace{0.1cm}\\ \frac{\mu_c(\cp_{\ell+1}(c))}{\zeta(2+\alpha)}
\frac{1}{(n-\ell)^{2+\alpha}}\frac{\mu_c(P)}{\mu_c(\cp_{n}(c)\setminus\cp_{n+1}(c))} & \text{ if }P\in\cp_{c,n}\text{ for }n> \ell\\
\end{cases},$$
where $\zeta(s)=\sum_{n=1}^\infty n^{-s} $ is the Riemann zeta function. Note that $\sum_{P\in\cp^*}m_{\ell}(P)=1$.
Moreover, taking
$$H(x)=\begin{cases}
0 & \text{ if }x=0\\
x\log(1/x) & \text{ if }x>0	
\end{cases},$$
we have that 
$$\sum_{n>\ell}R_c(P_{c,n})m_{\ell}(P_{c,n})=\sum_{n>\ell}(n+1)\,m_{\ell}(P_{c,n})=$$
$$=\sum_{n>\ell}(n+1)\sum_{P\in\cp_{c,n}}\frac{\mu_c(\cp_{\ell+1}(c))}{\zeta(2+\alpha)}
\frac{1}{(n-\ell)^{2+\alpha}}\frac{\mu_c(P)}{\mu_c(\cp_{n}(c)\setminus\cp_{n+1}(c))} =$$
$$=\frac{\mu_c(\cp_{\ell+1}(c))}{\zeta(2+\alpha)}\sum_{n>\ell}\bigg(
\frac{n+1}{(n-\ell)^{2+\alpha}}\frac{1}{\mu_c(\cp_{n}(c)\setminus\cp_{n+1}(c))} \underbrace{\sum_{P\in\cp_{c,n}}\mu_c(P) }_{\mu_c(\cp_{n}(c)\setminus\cp_{n+1}(c))}\bigg)=$$
$$=\frac{\mu_c(\cp_{\ell+1}(c))}{\zeta(2+\alpha)}\sum_{n>\ell}\frac{n+1}{(n-\ell)^{2+\alpha}}=\frac{\mu_c(\cp_{\ell+1}(c))}{\zeta(2+\alpha)}\sum_{n=1}\frac{n+1+\ell}{n^{2+\alpha}}\le$$
$$\le\frac{\mu_c(\cp_{\ell+1}(c))}{\zeta(2+\alpha)}(\ell+2)\sum_{n=1}\frac{1}{n^{1+\alpha}}=(\ell+1)\mu_c(\cp_{\ell+1}(c))\bigg(\frac{\ell+2}{\ell+1}\frac{\zeta(1+\alpha)}{\zeta(2+\alpha)}\bigg)\le$$
$$\le2(\ell+1)\mu_c(\cp_{\ell+1}(c))\frac{\zeta(1+\alpha)}{\zeta(2+\alpha)}\le 2(\ell+1)\zeta(1+\alpha)<\infty$$
and
$$\sum_{n>\ell}H(m_{\ell}(P_{c,n}))=\sum_{n>\ell}m_{\ell}(P_{c,n})\log(1/m_{\ell}(P_{c,n}))=$$

$$=\sum_{n>\ell}\sum_{P\in\cp_{c,n}}H\bigg(\frac{\mu_c(\cp_{\ell+1}(c))}{\zeta(2+\alpha)}
\frac{1}{(n-\ell)^{2+\alpha}}\frac{\mu_c(P)}{\mu_c(\cp_{n}(c)\setminus\cp_{n+1}(c))}\bigg)\le$$
$$=\sum_{n>\ell}\sum_{P\in\cp_{c,n}}H\bigg(
\frac{1/\zeta(2+\alpha)}{(n-\ell)^{2+\alpha}}\bigg)\le$$

$$\le\sum_{j>\ell}\frac{\log(\zeta(2+\alpha)(n-\ell)^{2+\alpha})}{\zeta(2+\alpha)(n-\ell)^{2+\alpha}}=$$
$$=\frac{1}{\zeta(2+\alpha)}\sum_{n=1}^{+\infty}\frac{\log(\zeta(2+\alpha)n^{2+\alpha})}{n^{2+\alpha}}\le$$
$$=\frac{\log(\zeta(2+\alpha))}{\zeta(2+\alpha)}\sum_{n=1}^{+\infty}\frac{1}{n^{2+\alpha}}+\frac{2+\alpha}{\zeta(2+\alpha)}\sum_{n=1}^{+\infty}\frac{\log(n)}{n^{2+\alpha}}=$$
$$=\log(\zeta(2+\alpha))+\frac{2+\alpha}{\zeta(2+\alpha)}\sum_{n=1}^{+\infty}\frac{n}{n^{2+\alpha}}\le$$
$$\le\log(\zeta(2+\alpha))+\frac{(2+\alpha)\zeta(1+\alpha)}{\zeta(2+\alpha)}\le3\zeta(1+\alpha)\le2(\ell+1)\zeta(1+\alpha)<\infty.$$
Taking $C=2(\ell+1)\zeta(1+\alpha)$, we have that $$\sum_{n>\ell}\sum_{P\in\cp_{c,n}}R_c(P)\, m_{\ell}(P)\le C\;\;\text{ and }\;\;\sum_{n>\ell}\sum_{P\in\cp_{c,n}}H(m_{\ell}(P))\le C.$$
Hence,  $$\sum_{P\in\cp^*}R_c(P) m_{\ell}(P)=\int_{\cp_{\ell}(c)} R_c d\mu_c+\bigg(\sum_{n>\ell}\sum_{P\in\cp_{c,n}}(n+1)m_{\ell}(P)\bigg)\le\int R_c d\mu_c+C<+\infty$$
and
$$\sum_{P\in\cp^*}H(m_{\ell}(P))=\bigg(\sum_{n=0}^{\ell}\sum_{P\in\cp_{c,n}}H(\mu_c(P))\bigg)+\bigg(\sum_{n>\ell}\sum_{P\in\cp_{c,n}}H(m_{\ell}(P_{c,n}))\bigg)\le$$
$$\le\bigg(\sum_{P\in\cp^*}H(\mu_c(P))\bigg)+\bigg(\sum_{n>\ell}\sum_{P\in\cp_{c,n}}H(m_{\ell}(P_{c,n}))\bigg)\le$$
$$\le h_{\mu_c}(F_c)+C<+\infty.$$

Thus, taking $\nu_{\alpha,\ell}$ as the ergodic $F_c$-invariant probability generated by the mass distribution $m_{\ell}$,
we get that, $$h_{\nu_{\alpha,\ell}}(F_c)=\sum_{P\in\cp^*}H(m_{\ell}(P))<+\infty$$
and
$$\int R_c d\nu_{\alpha,\ell}=\sum_{P\in\cp^*}R_c(P)\,m_{\ell}(P)\le\int R_c d\mu_c+C<+\infty.$$
It follows from Lemma~\ref{FolkloreResultB} that $$\eta_{\alpha,\ell}=\frac{1}{\int R_c d\nu_{\alpha,\ell}}\sum_{n\ge1}\sum_{j=0}^{n-1}F^j_*(\nu_{\alpha,\ell}|_{\{R_c=n\}})$$ is an ergodic $F$ invariant probability.
Note that $\supp\eta_{\alpha,\ell}=[c,p]$. Moreover, as $$F^j(\{R_c=n\})\subset(c,q)$$ for every $0\le j<n$ and $n\ge2$, we get that
\begin{equation}\label{Equation8765435678g}
\eta_{\alpha,\ell}|_{(q,p)}=\frac{1}{\int R_c d\nu_{\alpha,\ell}}\nu_{\alpha,\ell}|_{(q,p)} 
\end{equation}
For similar reason, as $\mu_c$ is the $F_c$-lift of $F$-invariant probability $\mu_0$, we have
\begin{equation}\label{Equation876356jhgfd}
\mu_0|_{(q,p)}=\frac{1}{\int R_c d\mu_c}\mu_c|_{(q,p)}.
\end{equation}

As $\int R_c d\nu_{\alpha,\ell}$ and $h_{\nu_{\alpha,\ell}}(F_c)<+\infty$, it follows from the generalized Abramov formula that $$h_{\eta_{\alpha,\ell}}(F)=\frac{h_{\nu_{\alpha,\ell}}(F_c)}{\int R_c d\nu_{\alpha,\ell}}.$$


On the other hand, using (\ref{Equation8765435678g}) and (\ref{Equation876356jhgfd}), we have that 
$$\int R d\eta_{\alpha,\ell}=\int_{(c,q)} R d\eta_{\alpha,\ell}+\int_{(q,p)}R d\eta_{\alpha,\ell}=t_0\eta_{\alpha,\ell}((c,q))+\int_{(q,p)}R d\bigg(\frac{1}{\int R_c d\nu_{\alpha,\ell}}\nu_{\alpha,\ell}|_{(q,p)} \bigg)=$$
$$=t_0\eta_{\alpha,\ell}((c,q))+\frac{1}{\int R_c d\nu_{\alpha,\ell}}\int_{J\setminus\cp_0(c)}R \,d\nu_{\alpha,\ell}=$$
$$=t_0\eta_{\alpha,\ell}((c,q))+\frac{1}{\int R_c d\nu_{\alpha,\ell}}\sum_{P\in\cp_{c,0}}R(P) m_{\ell}(P)=$$
$$=t_0\eta_{\alpha,\ell}((c,q))+\frac{1}{\int R_c d\nu_{\alpha,\ell}}\sum_{P\in\cp_{c,0}}R(P)\mu_c(P)=$$
$$=t_0\eta_{\alpha,\ell}((c,q))+\frac{1}{\int R_c d\nu_{\alpha,\ell}}\int_{(q,p)}R d\mu_c=$$
$$=t_0\eta_{\alpha,\ell}((c,q))+\frac{1}{\int R_c d\nu_{\alpha,\ell}}\int_{(q,p)}R\,d\bigg(\frac{1}{\int R_c d\mu_c}\mu_0|_{(q,p)}\bigg)=$$
$$=t_0\eta_{\alpha,\ell}((c,q))+\frac{1}{\int R_c d\nu_{\alpha,\ell}}\frac{1}{\int R_c d\mu_c}\int_{(q,p)}R\,d\mu_0<+\infty$$
Therefore,
$$\mu_{\alpha,\ell}=\frac{1}{\int R d\eta_{\alpha,\ell}}\sum_{n\ge1}\sum_{j=0}^{n-1}f^j_*(\eta_{\alpha,\ell}|_{\{R=n\}})$$
is an ergodic $f$ invariant probability and, as $f$ is transitive on $[f(c_+),f(c_-)]$ and $\supp\mu_{\alpha,\ell}\supset(c,p)$, we get that $$\supp\mu_{\alpha,\ell}=[f(c_+),f(c_-)].$$

\begin{Claim}\label{Claimljgfie976rfre18}
$\lim_{\ell}h_{\eta_{\alpha,\ell}}(F)\ge h_{\mu_0}(F)$
\end{Claim}
\begin{proof}[Proof of the claim]
Indeed, $$\lim_{\ell}h_{\eta_{\alpha,\ell}}(F)=\lim_{\ell}\frac{h_{\nu_{\alpha,\ell}}(F_c)}{\int R_c d\nu_{\alpha,\ell}}=\lim_{\ell\to+\infty}\frac{\sum_{P\in\cp^*}H(m_{\ell}(P))}{\sum_{P\in\cp^*}R_c(P)\,m_{\ell}(P)}=$$
$$=\frac{\lim_{\ell}\sum_{P\in\cp^*}H(m_{\ell}(P))}{\lim_{\ell}\sum_{P\in\cp^*}R_c(P)\,m_{\ell}(P)}=\frac{\sum_{P\in\cp^*}H(\mu_c(P))}{\int R_c d\mu_c}\ge\frac{h_{\mu_c}(F_c)}{\int R_c d\mu_c}=h_{\mu_0}(F).$$
\end{proof}

\begin{Claim}\label{Claimljgfiyre58}
	$\lim_{\ell}\int R\,d\eta_{\alpha,\ell}=\int R\,d\mu_0$
\end{Claim}
\begin{proof}[Proof of the claim]
Let $\widetilde{R}(x)$ be the $F_c$-lift of $R$, that is, $\widetilde{R}(x)=\sum_{j=0}^{R_c(x)-1}R\circ F^j(x)$.
We know that $\int R d\eta_{\alpha,\ell}=\frac{\int \widetilde{R} d\nu_{\alpha,\ell}}{\int R_c d\nu_{\alpha,\ell}}$ as well as 
$\int R d\mu_0=\frac{\int \widetilde{R} d\mu_c}{\int R_c d\mu_c}$.
As $\lim_{\ell}\int R_c d\nu_{\alpha,\ell}=\int R_c d\mu_c<+\infty$, we need to show that $\lim_{\ell}\int \widetilde{R} d\nu_{\alpha,\ell}=\int \widetilde{R} d\mu_c$.
To prove so, observe that $\widetilde{R}(x)$ is constant on each $P\in\cp_{c,n}$ and every $n\ge0$.
Indeed, $\widetilde{R}(x)=t_0 n+R(f^{t_0 n}(x))$ and $f^{t_0 n}(P)\in\cp_{c,0}=\cp\cap(q,p)$.
As $R$ is constant  on the elements of $\cp$, we get that $\widetilde{R}$ is constant on $P\in\cp_{c,n}$.
Therefore, 
$$\lim_{\ell}\int \widetilde{R} d\nu_{\alpha,\ell}=\lim_{\ell}\sum_{n=0}^{+\infty}\sum_{P\in\cp_{c,n}}\widetilde{R}(P)\nu_{\alpha,\ell}(P)=\sum_{n=0}^{+\infty}\sum_{P\in\cp_{c,n}}\widetilde{R}(P)\mu_c(P)=\int \widetilde{R} d\mu_c,$$
concluding the proof of the claim.
\end{proof}

Using Claim~\ref{Claimljgfie976rfre18} and Claim~\ref{Claimljgfiyre58} we can conclude that $\sup\{h_{\mu_{\alpha,\ell}}(f)\,;\,\ell\in\NN\}\ge h_{\mu}(f)$. Indeed, 
$$\lim_{\ell}h_{\mu_{\alpha,\ell}}(f)=\lim_{\ell}\frac{h_{\eta_{\alpha,\ell}}(F)}{\int R d\eta_{\alpha,\ell}}=\frac{\lim_{\ell}h_{\eta_{\alpha,\ell}}(F)}{\lim_{\ell}\int R d\eta_{\alpha,\ell}}=\frac{\lim_{\ell}h_{\eta_{\alpha,\ell}}(F)}{\int R d\mu_0}\ge \frac{h_{\mu_0}(F)}{\int R d\mu_0}=h_{\mu}(f).$$

As we can choose any $f$ invariant  probability $\mu$, with $\supp\mu=[f(c_+),f(c_-)]$, to cons\-truct $\mu_{\alpha,\ell}$ and as $\sup\{h_{\mu}(f)\,;\,\mu\in\cm^1(f)\text{ and }\supp\mu=[f(c_+),f(c_-)]\}=h_{top}(f)$,
we get that $$\sup\{h_{\mu}(f)\,;\,\mu\in\cm_{_{+\infty}}\}=h_{top}(f),$$
where $$\cm_{_{+\infty}}:=\{\mu_{\alpha,\ell}\,;\,\mu\in\cm^1(f)\text{ with }\supp\mu=[f(c_+),f(c_-)],\alpha\in(0,1)\text{ and }\ell\in\NN\}$$

Moreover, note that, if $0<\alpha_0<\alpha_1<1$, then $\nu_{\alpha_0,\ell}(P_{c,n})\ne\nu_{\alpha_1,\ell}(P_{c,n})$ for any $n>\ell$.
In particular, $\nu_{\alpha_0,\ell}\ne\nu_{\alpha_1,\ell}$.
As $R_c$ is exact, $F_c$ is orbit-coherent (see Lemma~2.6 of \cite{Pi20}).
Thus, it follows from Theorem~B of \cite{Pi20} that an ergodic $F$-invariant probability $\mu$ has at most one $F_c$-lift.
That is, setting $\cu=\{\nu\in\cm^1(F_c)$ $;$ $\nu$ is $F_c$ ergodic and $\int R_c d\nu<+\infty\}$, we have that $$\cu\ni\nu\mapsto\frac{1}{\int R d\nu}\sum_{n\ge1}\sum_{j=0}^{n-1}F^j_*(\nu|_{\{R_c=n\}})\in\cm^1(F)$$ is injective. 
Therefore, $\eta_{\alpha_{_0},\ell}\ne\eta_{\alpha_{_1},\ell}$ and, by the same argument, $\mu_{\alpha_{_0},\ell}\ne\mu_{\alpha_{_1},\ell}$.
This implies that $\cm_{_{+\infty}}=\{\mu_{\alpha,\ell}\,;\,\alpha\in(0,1)\text{ and }\ell\ge1\}$ is uncountable, and so the proof of items $(${\em i}$)$, $(${\em ii}$)$ and $(${\em iii}$)$ were obtained.

\textbf{ Item $(${\em iv}$)$}. To show that $\mu_{\alpha,\ell}$ has  fast recurrence to the singular region, we initially observed that $\int R_c d\eta_{\alpha,\ell}=+\infty$. 
Indeed, it is easy to see that $R_c$ is a exact induced time.
Hence, as $A$ is the domain of $R_c$ and $\eta_{\alpha,\ell}(A)=1$, it follows from Lemma~\ref{LemmaExacRtoRR} that $$\int R_c d\eta_{\alpha,\ell}\ge\frac{1}{2}\bigg(\int (R_c)^2d\nu_{\alpha,\ell}\bigg)\bigg/\bigg(\int R_c d\nu_{\alpha,\ell}\bigg)>$$
$$>\frac{1}{2(\int R_c d\mu_c+C)}\sum_{n>\ell}\sum_{P\in\cp_{c,n}}(n+1)^2\,m_{\ell}(P)=$$
$$=\frac{1}{2(\int R_c d\mu_c+C)}\sum_{n>\ell}\sum_{P\in\cp_{c,n}}\frac{\mu_c(\cp_{\ell+1}(c))}{\zeta(2+\alpha)}
\frac{(n+1)^2}{(n-\ell)^{2+\alpha}}\frac{\mu_c(P)}{\mu_c(\cp_{n}(c)\setminus\cp_{n+1}(c))}=$$
$$=\frac{1}{2(\int R_c d\mu_c+C)}\frac{\mu_c(\cp_{\ell+1}(c))}{\zeta(2+\alpha)}\sum_{n>\ell}\frac{(n+1)^2}{(n-\ell)^{2+\alpha}}\underbrace{\sum_{P\in\cp_{c,n}}
\frac{\mu_c(P)}{\mu_c(\cp_{n}(c)\setminus\cp_{n+1}(c))}}_1=$$
$$=\frac{1}{2(\int R_c d\mu_c+C)}\frac{\mu_c(\cp_{\ell+1}(c))}{\zeta(2+\alpha)}\sum_{n>\ell}\frac{(n+1)^2}{(n-\ell)^{2+\alpha}}=$$ 
$$=\frac{\mu_c(\cp_{\ell+1}(c))}{2(\int R_c d\mu_c+C)\zeta(2+\alpha)}\sum_{n=1}\frac{(n+1+\ell)^2}{n^{2+\alpha}}\ge\frac{\mu_c(\cp_{\ell+1}(c))}{2(\int R_c d\mu_c+C)\zeta(2+\alpha)}\sum_{n=1}\frac{1}{n^{\alpha}}=+\infty.$$

\begin{Claim}
There are $K>0$ and $x_0\in(c,q)$ such that $|\log|x-c||\ge K R_c(x)$ for every $x\in(c,x_0)$.
\end{Claim}
\begin{proof}[Proof of the claim]
As $f$ is non-flat, there is $1<\theta_1\le\theta_2$ and $a\ge1$ such that $$(1/a)(x-c)^{1/\theta_1}\le f(x)-f(c_+) \le a(x-c)^{1/\theta_2}$$
for every $x>c$. As $f^{t_0}(c_+)=c$ and $0<f'(f^j(c_+))<+\infty$ for every $1\le j<t_0$, there exist $b\ge 1$ and $\delta_0>0$ such that 
$$(1/b)(x-c)^{1/\theta_1}\le f^{t_0}(x)-c \le b(x-c)^{1/\theta_2}.$$
Taking $1<\theta_1'<\theta_1\le\theta_2<\theta_2'$ and $0<\delta_0'<\delta_0$ small, we get that 
$$(x-c)^{1/\theta_1'}\le f^{t_0}(x)-c \le (x-c)^{1/\theta_2'}.$$
for every $c<x<c+\delta_0'$.
Hence, taking $g=(f^{t_0}|_{(c,c+\delta_0')})^{-1}$, we have 
$$(x-c)^{\theta_2'}\le g(x)-c \le (x-c)^{\theta_1'}$$ and so
$$(x-c)^{(\theta_2')^n}\le g^n(x)-c \le (x-c)^{(\theta_1')^n}.$$

Let $n_0$ be the smaller $n\ge1$ such that $P_{c,n}\subset(c,f(c+\delta_0'))$.
Taking $\gamma=\sup P_{c,n_0}$, we get that $P_{c,n}=[g^{n-n_0+1}(\gamma),g^{n-n_0}(\gamma))$ for every $n\ge n_0$.
Thus, $(\gamma-c)^{\theta_2'}\le|x-c|\le(\gamma-c)$ for $x\in P_{c,n_0}$ and 
$$(\gamma-c)^{(\theta_2')^{(n-n_0+1)}}\le|x-c|\le(\gamma-c)^{(\theta_1')^{(n-n_0)}},$$
For every $x\in P_{c,n_0}$ and $n>n_0$.
This means that 
$$\log(1/(\gamma-c))(\theta_1')^{(n-n_0)}\le|\log|x-c||\le\log(1/(\gamma-c))(\theta_2')^{(n-n_0+1)},$$
$x\in P_{c,n}$ and $n\ge n_0$.
As $R_c(P_{c,n})=n$, if we take $n_1\ge n_0$ such that $(\theta_1')^n\ge n$ for every $n\ge n_1$,   $K=\frac{1}{2(\theta_1')^{n_1}}\log(\frac{1}{\gamma-c})$ and $x_0=\sup P_{c,n_0+1}=\inf P_{c,n_0}$ then 
$$|\log|x-c||\ge K\,{\theta_1'}^{R_c(x)}\ge K R_c(x)$$
for every $c<x<x_0$. 
\end{proof}

As $$\int_J R_c d\mu_{\alpha,\ell}\ge\frac{1}{\int R d\eta_{\alpha,\ell}} \int R_c d\eta_{\alpha,\ell}=+\infty$$ and $c$ is the unique pole of $R_c$, we conclude that $\int_{x\in(c,x_0)}R_c(x)d\mu_{\alpha,\ell}=\infty$ and so, 
\begin{eqnarray*}
\int_{x\in[0,1]}|\log|x-c||d\mu_{\alpha,\ell}&\ge & \int_{x\in(c,x_0)}|\log|x-c||d\mu_{\alpha,\ell}\ge \\
&\ge & K\int_{x\in(c,x_0)}R_c(x)d\mu_{\alpha,\ell}=+\infty,
\end{eqnarray*} proving item $(${\em iv}$)$.

Finally for item $(${\em v}$)$,  it follows from $f$ be non-flat that there are constants $c_0,c_1,c_2>0$ such that $-c_0+c_1|\log|x-c||\le \log|f'(x)|\le c_0+c_2|\log|x-c||$ for every $x\in[0,1]\setminus\{c\}$ and so $\int\log|f'|d\mu_{\alpha,\ell}=+\infty$.
Hence, it follows from Birkhoff and the ergodicity of $\mu_{\alpha,\ell}$ that 
$\lim_n\frac{1}{n}\log(f^n)'(x)=\lim_n\frac{1}{n}\sum_{j=0}^{n-1}\log(f'\circ f^j(x))=\int\log|f'|d\mu_{\alpha,\ell}=+\infty$ for $\mu_{\alpha,\ell}$ almost every $x\in[0,1]$. 
\end{proof}

\begin{proof}[Proof of Corollary~\ref{MainCorAA}]
According to the Remark~\ref{RemarkDense} (see also Lemma~\ref{FolkLemma} in Appendix below), for every $f\in\cl(c,\alpha,\beta)$, one can find $g\in\cl(c,\alpha,\beta)$ arbitrarily close to $f$ satisfying the hypothesis of Theorem~\ref{MainTheoA}.
In the proof of Theorem~\ref{MainTheoA}, the construction of an $g$-probability $\mu$ with infinity Lyapunov exponent and high entropy is based on a full induced map $F$ for $g$.
Let $\mu$ be an ergodic $g$-invariant probability with, say, $h_{\mu}(g)\ge h_{top}(g)-1/(2k)$ and $\int\log g' d\mu\ge2\lambda$, for some $k\in\NN$.
This $\mu$ can be arbitrarily approximated by measures $\mu_n$ supported in $\{R\le n\}$, where $R$ is the induced time of $F$.
Moreover, we have $h_{\mu_n}(g)\ge h_{top}(g)-1/k$ and $\int\log g' d\mu_n\ge\lambda$ for $n$ big enough.
As $F|_{\{R\le n\}}$ is a horseshoe (indeed, $\Lambda=\bigcap_{j\ge0}F^{-j}(\{R\le n\})$ is transitive maximal invariant uniform expanding set), for every $h\in\cl(c,\alpha,\beta)$ close enough to $g$, there is an analytic continuation of $F|_{\{R\le n\}}$ for any fixed $n$.
In particular, an ``analytic continuation'' $\widetilde{\mu}_n$ of $\mu_n$.
Therefore, given $\lambda>0$ and $k\in\NN$, one can find a neighborhood $\cn_k(f)$ of $g$ such that $\sup\{h_{\mu}(h)\,;\,\int\log h'd\mu\ge\lambda\}\ge h_{top}(h)-1/(2k)$ for every $h\in\cn_k(f)$. 
Note that $\cn_k(f)$ is an open set close to $f$. 
Hence,  $\cR(\lambda)=\bigcap_{k\in\NN}\bigcup_{f\in\cl(c,\alpha,\beta)}\cn_k(f)$ is a residual set of $\cl(c,\alpha,\beta)$ in which $h_{top}(h)=\sup\{h_{\mu}(h)\,;\,\int\log h'd\mu\ge\lambda\}$ for every $h\in\cR(\lambda)$.
Therefore, $\cR=\bigcap_{n\in\NN}\cR(n)$ is also residual, concluding the proof.
\end{proof}

\section{Appendix}
\begin{Lemma}[Folklore lemma]\label{FolkLemma}
$\cl_{\per}(c,\alpha,\beta)$ is dense in $\cl(c,\alpha,\beta)$.
\end{Lemma}
\begin{proof}[Sketch of the proof]
Given $f\in\cl(c,\alpha,\beta)$ and $\varepsilon>0$, let $J=(a,b)$ be a nice interval containing $c$. That is, $\co_f^+(f(a))\cap J=\emptyset=\co_f^+(f(b))\cap J$ $($\footnote{ The existence of nice intervals containing singularity with arbitrarily small diameter is well know.
For instance, as $\per(f)$ is dense in $[f(c_+),f(c_-)]$, taking $p\in\per(f)\cap(c-\varepsilon,c)$ and $q\in\per(f)\cap(c,c+\varepsilon)$, we get that $(a,b)$ is a nice interval containing in $(c-\varepsilon,c+\varepsilon)$, where $a=\max \big(\co_f^+(p)\cup\co_f^+(q)\big)\cap(0,c)$ and $b=\min\big(\co_f^+(p)\cup\co_f^+(q)\big)\cap(c,1)$.}$)$.
It is easy to check that $\Lambda=\{x\in[0,1]\,;\,\co_f^+(x)\cap J=\emptyset\}$ is an uniformly expanding set.
Moreover, changing the nice interval if necessary, we may assume that $\Lambda$ is a Cantor set.
One can use the {\em homterval lemma} (see for instance Lemma~3.1 at \cite{dMvS}) to verify that for each gap $I$ of $\Lambda$, i.e., a connected component of $[0,1]\setminus\Lambda$, there exists a unique $n(I)\ge0$ such that
\begin{equation}\label{Equationjhfyf}
f^{n(I)}(I)=J\text{ and }f^{j}(I)\cap J=\emptyset\;\;\forall\,0\le j<n(I).
\end{equation}

First assume that exists a gap $R$  of $\Lambda$ such that $f(c_-)\in\overline{R}$ and, by \eqref{Equationjhfyf}, set $r:=n(R)$ (see Figure~\ref{Deslocamento1.png}).
\begin{figure}
  \begin{center}\includegraphics[scale=.3]{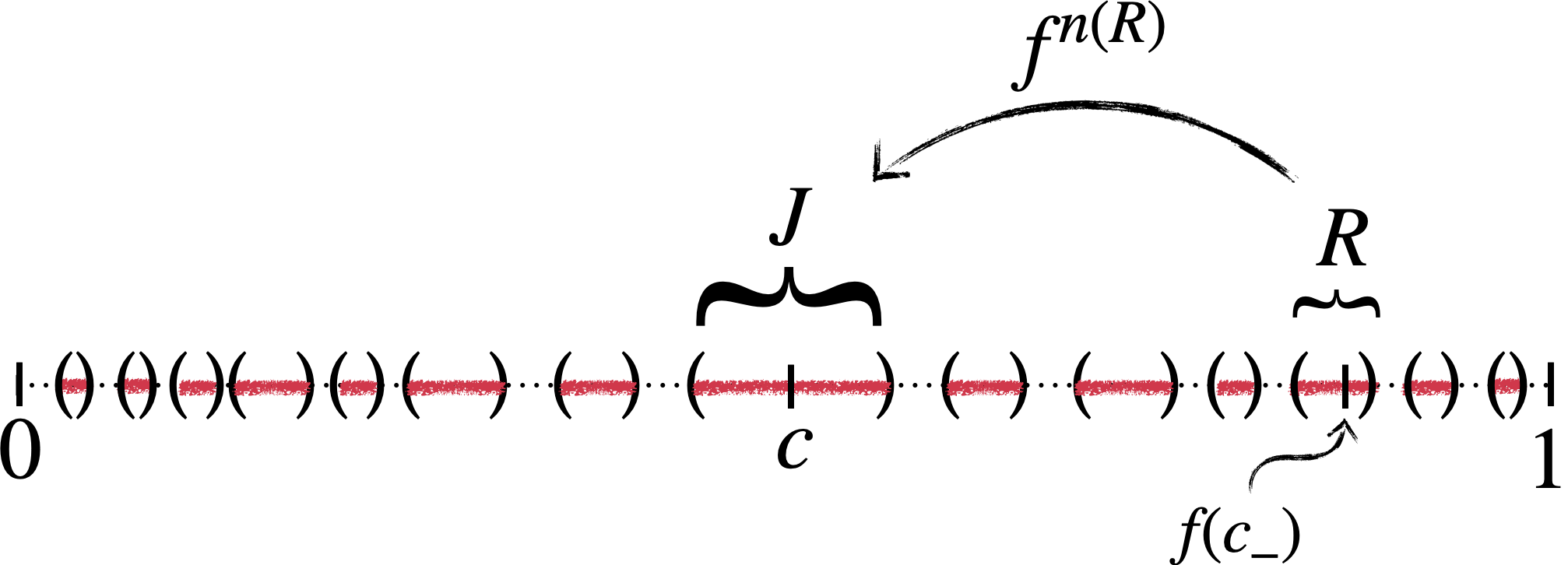}\\
  \caption{\small In this figure on can see the gaps of the Cantor set $\Lambda=\{x\,;\,\co_f^+(x)\cap J=\emptyset\}$, where $J$ is a nice interval containing $c$.
  The gap $R$ contains the singular value $f(c_-)$ and $f^{n(R)}$ sends $R$ onto $J$ diffeomorphically.}\label{Deslocamento1.png}
  \end{center}
\end{figure}
As in Figure~\ref{Perturbacao.png}, one can smoothly perturb $f$ to get a map $g\in\cl(c,\alpha,\beta)$ such that $g(x)=f(x)$ for $x\notin J\cap(0,c)$ and $g(c_-)=(f^r|_{R})^{-1}(c)$.
\begin{figure}
  \begin{center}\includegraphics[scale=.4]{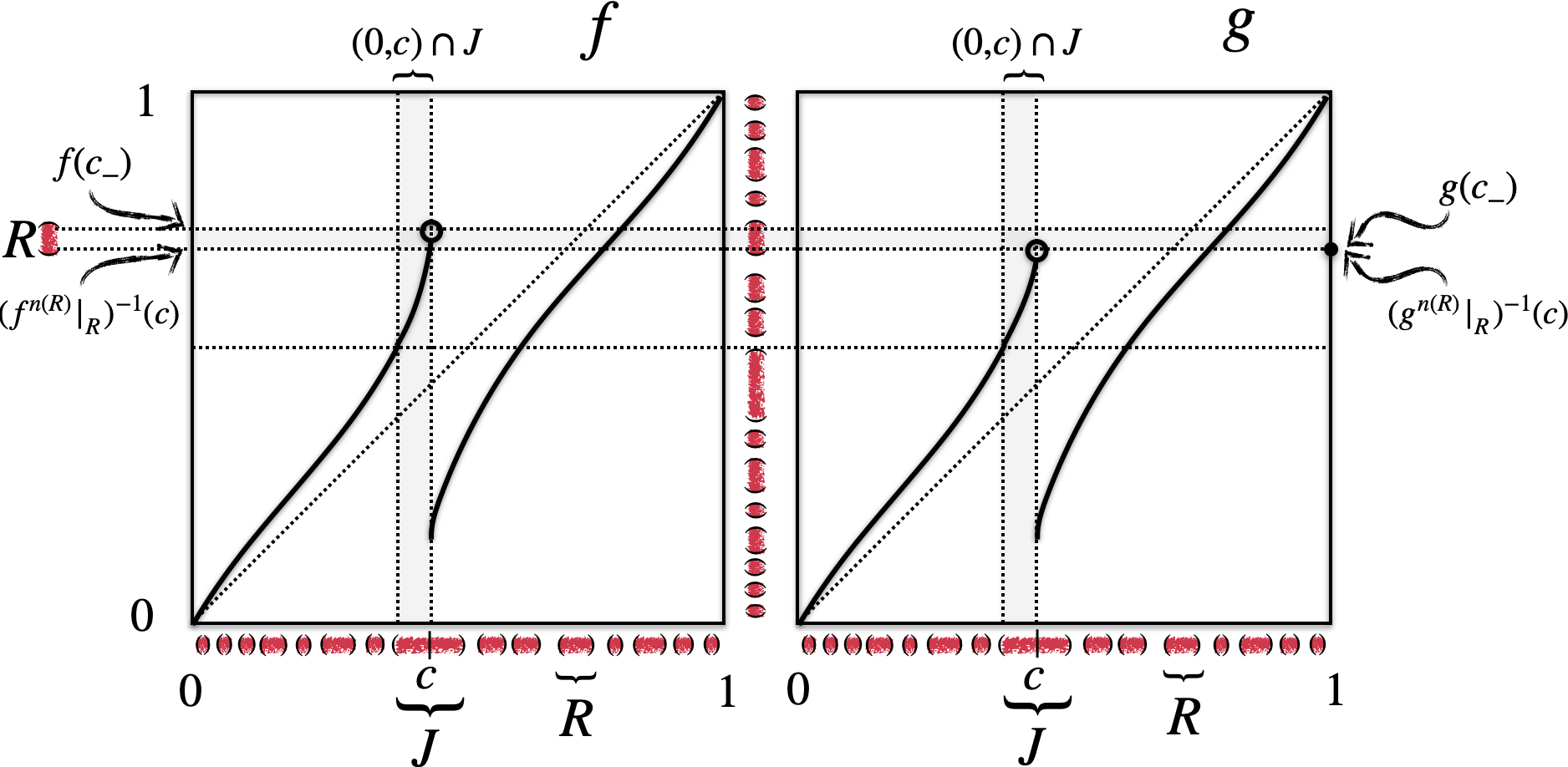}\\
  \caption{\small On the left side we have the initial Lorenz map $f$ (the unperturbed on).
  On the right, one can see the map $g$ close to $f$ having this map is extended to present two fixed points at the end of the interval.}\label{Perturbacao.png}
  \end{center}
\end{figure}
Therefore, $g^{r+1}(c_-)=c$.  
As $r=n(R)\to+\infty$ when $\varepsilon\to0$, we get that $\gamma:=\inf\{(f^r)'(x)\,;\,x\in R\}\to+\infty$.
In particular $|f(c_-)-g(c_-)|\le |R|\le\frac{1}{\gamma}|J|\le\varepsilon$ for $\varepsilon$ small enough.
This also ensures that $\dist(f,g)\to0$ when $\varepsilon\to0$.

If $f(c_-)$ does not belongs to the closure of a gap, then there exists a sequence $R_k$ of gaps of $\Lambda$ such that $\diameter(R_k)\to0$ and $d(f(c_-),R_k):=\inf\{|f(c_-)-x|\,;\,x\in R_k\}\to0$ (see Figure~\ref{Deslocamento2.png}).
\begin{figure}
  \begin{center}\includegraphics[scale=.3]{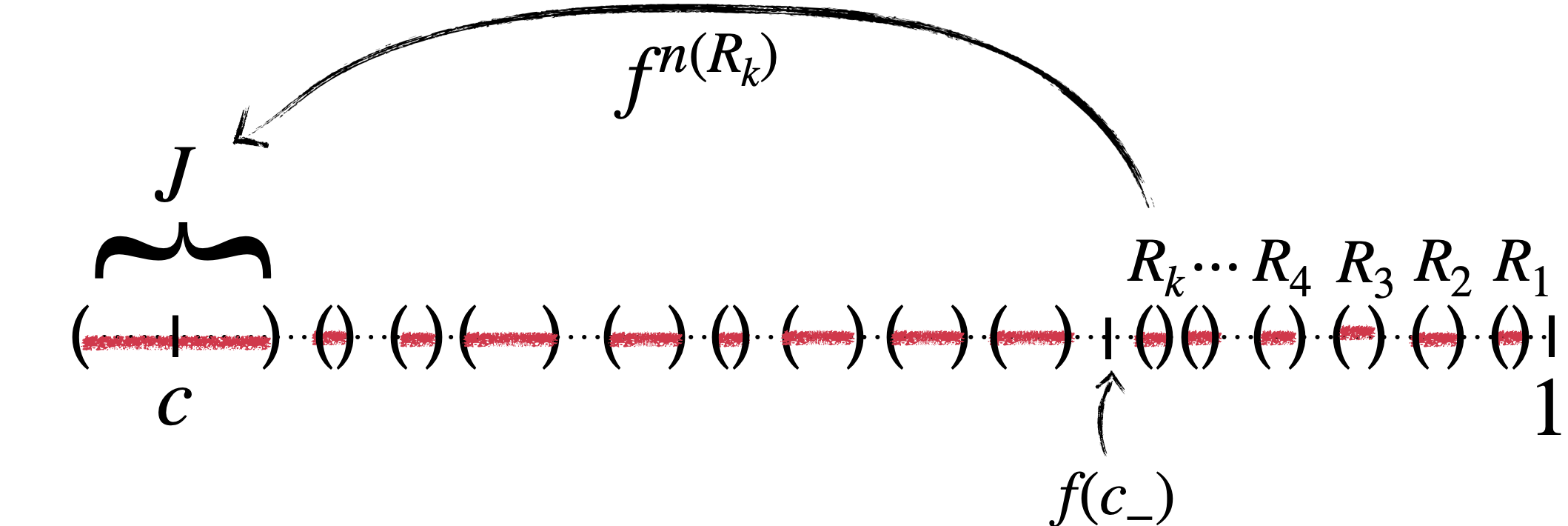}\\
  \caption{\small This figure is similar to Figure~\ref{Deslocamento1.png}, but here the singular value $f(c_-)$ is not contained in the closure of a gap of $\Lambda$. In this case, $f(c_-)$ is accumulated by a sequence $R_k$ of gaps of $\Lambda$.}\label{Deslocamento2.png}
  \end{center}
\end{figure}
Thus, taking $k$ so that $d(f(c_-),R_k)<\varepsilon$ and $g$ as before, that is, $g(x)=f(x)$ for $x\notin J$ and $g(c_-)=(f^r|_{R_k})^{-1}(c)$, we have again that $g^{r+1}(c_-)=c$ and that $\dist(f,g)$ can be as small as we want when $\varepsilon$ goes to zero.

After perturbing the left side of the singularity to obtain $g$ close to $f$ with $g^{r+1}(c_-)=c$, we can, by it turn, perturb $g$ on the right side of $c$ to obtain $h$ close to $g$, and thus close to $f$ and such that $h^{\ell+1}(c_+)=c$ and $h^{r+1}(c_-)=g^{r+1}(c_-)=c$.
Finally, taking $m=(r+1)(\ell+1)$, we have that $h$ is close to $f$ and $h^m(c_-)=c=h^m(c_+)$, completing the proof.
\end{proof}


\begin{thebibliography}{10}



\bibitem[ABS]{ABS} Afraimovich, V.S., Bykov, V.V., Shil'nikov, L.P.. {\em On the appearance and structure of the Lorenz attractor}. Dokl. Acad. Sci. USSR, 234: 336-339, 1977.


\bibitem[AP]{AP} Araújo, V.; Pacifico, M. J.. {\em Three-Dimensional Flows}. Springer, ISBN: 978-3-642-11414-4, 2010.

\bibitem[Do]{Do} Dobbs, N.. {\em On cusps and flat tops}. Ann. Inst. Fourier (Grenoble) 64(2), 571-605, 2014.

\bibitem[GH]{GH}
Guckenheimer, J.; Holmes, P.. {\em Nonlinear Oscillations, Dynamical Systems, and Bifurcations of Vector Fields}, Springer,
ISBN: 978-1-4612-7020-1, 1983.

\bibitem[GW]{GW} Guckenheimer, J., Williams, R.F..{\em Structural stability of Lorenz attractors}. Publ. Math. IHES,
50, 59-72, 1979.

\bibitem[Lo]{L} Lorenz, E.N. {\em  Deterministic nonperiodic flow}. J. Atmosph. Sci., 20:130-141, 1963.



\bibitem[dMvS]{dMvS} de Melo, W.; van Strien, S.. {\em One Dimensional Dynamics}. Springer-Verlag, 1993.

\bibitem[OV]{OV}
Olivares-Vinales, {\em Invariant measures for interval maps without Lyapunov exponents}. Ergodic Theory and Dynamical Systems, 1-29. doi:10.1017/etds.2021.128, 2021.

\bibitem[Pe]{Pe} Petersen, K..{\em Ergodic Theory}. Cambridge Studies in Advanced Mathematics 2. Cambridge University Press, 1990.

\bibitem[Pi11]{Pi11} Pinheiro, V.. {\em Expanding Measures}. Annales de l Institut Henri Poincaré. Analyse non Linéaire, v. 28, p. 889-939, 2011.

\bibitem[Pi20]{Pi20} Pinheiro, V.. {\em Lift and Synchronization}. ArXiv: 1808.03375v3 [math.DS], 1-54, 2020.


\bibitem[Pr]{Pr}
Przytycki, F. . {\em Lyapunov Characteristic Exponents are nonnegative}. Proceedings of the American Mathematical Society, 119, 309-317, 1993.

\bibitem[Sp]{Sp} Sparrow, C..
{\em The Lorenz Equations: Bifurcations, Chaos, and Strange Attractors}, Springer, ISBN: 978-0-387-90775-8, 1982.


\bibitem[Tu]{T} Tucker, W.. {\em The Lorenz attractor exists}. C.R. Acad. Sci. Paris, 328, Série I:1197-1202, 1999.


\bibitem[Wi]{Wi}  Williams, R.F..{\em The structure of Lorenz attractors}. Publ. Math. IHES, 50:73-99, 1979.

\end{thebibliography}
\end{document}